\documentclass[11pt]{article}

\usepackage{graphicx}

\usepackage{macrosetc}

\title{Moduli Spaces of Morse Functions for Persistence}
\author{Michael J. Catanzaro\thanks{Iowa State University,  mjcatanz@iastate.edu}
\and  Justin M. Curry\thanks{University at Albany SUNY,  jmcurry@albany.edu}
\and Brittany Terese Fasy\thanks{Montana State University,  brittany.fasy@montana.edu}
\and J\=anis Lazovskis\thanks{University of Aberdeen, janis.lazovskis@abdn.ac.uk}
\and Greg Malen\thanks{Duke University, gmalen@math.duke.edu}
\and Hans Riess\thanks{University of Pennsylvania, hmr@seas.upenn.edu}
\and Bei Wang\thanks{University of Utah, beiwang@sci.utah.edu.}
\and Matthew Zabka\thanks{Southwest Minnesota State University,  Matthew.Zabka@smsu.edu}}
\date{\today}

\begin{document}

\maketitle
\begin{abstract}
We consider different notions of equivalence for Morse functions on the sphere
in the context of persistent homology and introduce new invariants to study
these equivalence classes. These new invariants are as simple---but more
discerning than---existing topological invariants, such as persistence barcodes
and Reeb graphs. We give a method to relate any two Morse--Smale vector fields
on the sphere by a sequence of fundamental moves by considering graph-equivalent
Morse functions. We also explore the combinatorially rich world of
height-equivalent Morse functions, considered as height functions of embedded
spheres in $\Rspace^3$. Their level set invariant, a poset generated by nested
disks and annuli from level sets, gives insight into the moduli space of Morse
functions sharing the same persistence barcode.

\end{abstract}


\section{Introduction}
\label{sec:introduction}

Morse theory describes the topology of a manifold $\Mspace$ by studying well-behaved functions $f: \Mspace \to \Rspace$ \cite{Milnor1963}. This \emph{well-behavedness} is qualified by the notion of a \emph{Morse function}: a smooth real-valued function with no degenerate critical points. One of the main ideas of Morse theory is to associate the topological changes of the sublevel sets $\Mspace_a = f^{-1}(\infty, a]$, as $a$ varies, with the critical points of $f$.
From an algebraic perspective, Morse functions are \emph{effective} in topological problems due to their local rigidity; that is, critical points of Morse functions have a very simple local, quadratic structure (up to a change of coordinates) \cite{Nicolaescu2007}. From a homological perspective, the relationship between the topology of $\Mspace$ and the critical points of $f$ is described by the powerful Morse inequalities, which have both topological and geometric significance.

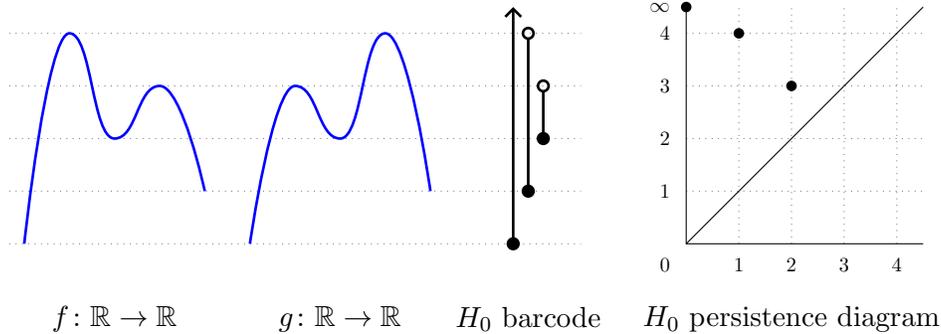
\begin{figure}[!ht]\centering
\begin{tikzpicture}
\foreach \y in {0,...,4}{
  \draw[opacity=.5,dotted] (0,\y*.7)--(7.6,\y*.7);
}
\begin{scope}[shift={(.2,0)}]
\draw[line width=1pt,blue,yscale=.7,xscale=.6] (0,0) .. controls +(0:0) and +(180:.5) .. (1,4)
            .. controls +(0:.5) and +(180:.5) .. (2,2)
            .. controls +(0:.5) and +(180:.5) .. (3,3)
            .. controls +(0:.5) and +(0:0) .. (4,1);
\node at (2*.6,-1) {$f\colon \Rspace\to \Rspace$};
\end{scope}
\begin{scope}[shift={(3.2,0)}]
\draw[line width=1pt,blue,yscale=.7,xscale=.6] (0,0) .. controls +(0:0) and +(180:.5) .. (1,3)
            .. controls +(0:.5) and +(180:.5) .. (2,2)
            .. controls +(0:.5) and +(180:.5) .. (3,4)
            .. controls +(0:.5) and +(0:0) .. (4,1);
\node at (2*.6,-1) {$g\colon \Rspace\to \Rspace$};
\end{scope}
\begin{scope}[shift={(6.7,0)},yscale=.7] 
\draw[barcinf] (0,0)--(0,4.5);
\draw[barco] (.2,1)--(.2,4);
\draw[barco] (.4,2)--(.4,3);
\node at (.2,-1/.7) {$H_0$ barcode};
\end{scope}
\begin{scope}[shift={(9,0)},scale=.7]   
\draw (0,0)--(4.5,4.5);
\foreach \x in {1,...,4}{
  \draw[opacity=.5,dotted] (\x,0)--(\x,4.5);
  \draw[opacity=.5,dotted] (0,\x)--(4.5,\x);
  \node[scale=.7] at (\x,-.4) {\x};
  \node[scale=.7] at (-.4,\x) {\x};
}
\draw (0,4.5)--(0,0)--(4.5,0);
\node[scale=.7] at (-.4,-.4) {0};
\node[scale=.7] at (-.5,4.5) {$\infty$};
\fill (0,4.5) circle (.1);
\fill (1,4) circle (.1);
\fill (2,3) circle (.1);
\node at (2,-1/.7) {$H_0$ persistence diagram};
\end{scope}
\end{tikzpicture}
\caption{The graphs of two different functions $f, g: \Rspace \to \Rspace$
    restricted to a finite interval in $\Rspace$, and their identical
    one-dimensional barcode and persistence diagram based on sublevel set
    filtration. The longest bar in the barcode captures the connected component;
    while the 2nd longest bar is created due to the boundary condition.}
    \label{fig:equivalent-barcodes}
\end{figure}

Persistent homology is a relatively new tool for discriminating functions on topological spaces based on how the shape of their (sub)level sets evolve. In the standard setting, persistence is an extension of Morse theory, as it studies homology groups of sublevel sets connected by inclusion maps, $\Mspace_a \xhookrightarrow{} \Mspace_b$ (for $a \leq b$). The evolution of shape is captured by what is known as the \emph{persistence diagram} or \emph{barcode}~\cite{EdelsbrunnerHarer2010,Ghrist2008}. Barcodes enjoy properties such as \emph{simplicity}, as a barcode is simply a collection of intervals in the real line; and \emph{stability}, as small perturbations of shape produce small perturbations of the barcode. Both of these properties make persistence an ideal tool for studying the shape of data, with wide applications to science and engineering; see \cite{EdelsbrunnerMorozov2012} for a survey.

We are interested in exploring different moduli spaces of Morse functions from the perspective of persistence. We characterize the set of Morse functions that give rise to the same barcode, see Figure~\ref{fig:equivalent-barcodes} for an example. Some of these functions should be considered as different, because taking one function to another requires a significant deformation. In other words, we are putting different equivalence relations on the space of Morse functions that respect persistence, i.e.,~two functions can only be deemed equivalent only if they have the same barcode, but simply having the same barcode does not guarantee equivalence.  Each choice of equivalence relation leads to a different moduli space structure on the space of Morse functions; and each equivalence class has an  interesting combinatorial structure that can be used practically to enrich the barcode.

Instead of focusing on \emph{any} Morse function $f\colon \Mspace \to \Rspace$, we are initially motivated by a simpler question by considering $\Mspace = \Sspace^2$: \emph{How many Morse functions on the sphere~$\Sspace^2$ have the same barcode?} Or more precisely: \emph{How many equivalence classes of Morse functions on the sphere have the same barcode?} Asking such an open question is a first step towards exploring three areas of interests described below. For simplicity, we sometimes require more from $f$, such as factoring as an embedding into Euclidean space followed by a projection, so as to exclude pathological examples like the Alexander horned sphere.

\textbf{Topological data analysis:} The simplicity of the barcode is both a
benefit and a drawback, as information about the space and function is lost
during its computation. Precisely how much information is lost?  How does the
barcode compare to other topological invariants, and does it fail to capture
some key topological and geometric features? How does the analysis on a sphere
extend to a compact surface of genus $g$?  From a statistical perspective, how
common is a particular barcode, and which one should we expect in a
\emph{general} situation?

\textbf{Shape analysis:} Properties of shapes, such as being convex, nested,
elongated, or circular, are mostly geometric in nature.   Do the discriminative
capabilities of persistence---known to be mostly topological---also capture these
geometric properties? Recent work by \cite{BubenikHullPatel2019} has shown that
short intervals in barcodes encode geometric information; in particular,
persistent homology detects the curvature of disks from which points are
sampled.  How much geometric information is preserved by the barcodes?

Topological descriptors, such as Reeb graphs~\cite{Reeb46}
and Morse--Smale complexes~\cite{EdelsbrunnerHarerNatarajan2003,EdelsbrunnerHarerZomorodian2003},
provide
an abstract and compact representation of data modeled by Morse
functions. For
instance, the small number of cells in a Morse--Smale complex can significantly
reduce the number of cells when discretizing a shape, while keeping the same
topological properties. Does persistence allow us to reduce the carried data
even further? 

We refer the interested reader to the survey paper by \cite{biasotti2008describing} for a summary of shape analysis's relationship to persistence and Reeb graphs. 

\textbf{Dynamical systems:}
What is the space of all Morse or Morse--Smale vector fields, with and without the requirement of identical persistence as a constraint? Equivalence classes of Morse--Smale vector fields on 2-manifolds have been studied previously \cite{Peixoto1973,Fleitas1975,GutierrezMelo1977}; however, not in the context of persistence. It is known that the set of Morse--Smale vector fields on orientable surfaces is dense \cite[Theorem 2.6]{PalisMelo1982}. Given two Morse--Smale vector fields which are not topologically equivalent, can we derive a distance measure between them? What is the minimal number of operations (critical pair cancellation or reverse-cancellation) to transform one to another?

\subsection{Research Objectives}
Our overarching goal is to classify and algorithmically construct all equivalence classes of Morse functions on a manifold, where the equivalence is captured by functions, embeddings, or dynamics, using barcodes as constraints. Fixing a homology degree, the \emph{persistence map} is the map that takes a function $f$ on $\Mspace$ to its associated barcode~\cite{Curry2017}.  We classify the image, preimage, and embeddings of the preimage of the persistence map under different notions of equivalence relations. While there have been approaches to interpreting the persistence map functorially \cite{BauerLesnick2015,SilvaMunchPatel2016}, and such categorical generalizations are of interest to us (see Section \ref{sec:results-posets}), considering the persistence map simply as a map rather than a functor does not take away from our analysis.

Let $\Mspace$ denote a smooth manifold, and let $f\colon \Mspace \to \Rspace$ denote a Morse function.  In this paper, we focus on characterizing Morse functions on the sphere $\Mspace = \Sspace^2$ that produce the same barcode.

\begin{itemize}

\item \textbf{Objective 1: Classifying the image.} Fix $\Mspace$ (for example,~$\Sspace^2$) and the number and type of critical points that respect the Euler characteristic (for example,~two maxima, one saddle point, and one minimum on $\Sspace^2$). Enumerate the barcodes that correspond to sublevel set filtrations of functions $f\colon \Mspace \to \Rspace$.  This is a computational objective, and can be interpreted by counting Morse--Smale graphs (Section~\ref{sec:background}) via elementary moves (Section~\ref{sec:results-abcd}) instead of functions. In this setting, we declare that two  Morse functions $f, g\colon \Mspace \to \Rspace$ are \emph{indistinguishable} if they are \emph{graph equivalent} (Section~\ref{sec:background}).

\item \textbf{Objective 2: Classifying embeddings of the preimage.} Fix $\Mspace$ and a barcode, and ask how many different embeddings $\iota: \Mspace \to \Rspace^d$ of the space $\Mspace$ into Euclidean space $\Rspace^d$ produce the given barcode when projected onto a fixed axis via $\pi \colon \Rspace^d \to \Rspace$. In other words, we study the behavior of the function $f:= \pi \circ \iota$. This is related to the persistent homology transform~\cite{TurnerMukherjeeBoyer2014}, which asks the same question, but for all directions. Here, two Morse functions $f, g: \Mspace \to \Rspace$ (where $f\colonequals \pi \circ \iota$ and $g\colonequals \pi \circ \iota'$) are considered \emph{indistinguishable} if (i) they generate the same barcode and (ii) they are \emph{poset equivalent} (Section~\ref{sec:background}).

\item \textbf{Objective 3: Classifying the preimage.} Fix $\Mspace$ and a barcode, and ask for all functions, up to \emph{level-set preserving equivalence}, whose image is the given barcode. This topological objective benefits from the Reeb graph~\cite{Reeb46}, which enriches the barcode by distinguishing different type of persistence pairings, and its refinement the \emph{decorated Morse--Smale graph}, introduced in Section \ref{sec:morsefuncs-persistence}. Two Morse functions $f, g\colon \Mspace \to \Rspace$ are  \emph{indistinguishable} in this context if (i) they generate the same barcode and (ii) they give rise to isomorphic decorated Morse--Smale graphs.
\end{itemize}

Classifying Morse functions under different notions of equivalences via the persistence map is well-motivated, as persistence has emerged as a central tool of topological data analysis. Applications of persistence include shape analysis~\cite{CarlssonZomorodianCollins2004,PoulenardSkrabaOvsjanikov2018}, cancer research~\cite{SeemannShulmanGunaratne2012,LockwoodKrishnamoorthy2015,HofmannKrufczikHeermann2018,QaiserTsangTaniyama2019} and material sciences~\cite{LeeBarthelDlotko2018}. By exploring the above moduli spaces, we aim to build a better or enriched  barcode for real-world applications.

\subsection{Related Work}

The mathematical study of height functions through level sets (contour lines) and
relationships between them dates back to at least the 1850s, when Cayley
classified topographical maps based on configurations of contour lines~\cite{cayley1859xl}. 
\cite{maxwell1870hills} extended this work and laid the foundation for
Morse theory.
\cite{Reeb46} defined a graphical invariants to classify Morse functions; today, we call
this invariant the \emph{Reeb graph}.
Later, \cite{Arnold1991,Arnold1992} defined geometric equivalence
classes of functions on $\Sspace^1$ using a notion of ``snakes". The
classification for functions on surfaces is much more involved, and has been
considered in certain cases~\cite{Arnold2007,Nicolaescu2007}, the latter of which analyzed
homological and geometric equivalence of Morse functions on $\Sspace^2$.
Further, ~\cite{Kulinich1998} and ~\cite{Sharko1996,Sharko2003} classified Morse functions
on surfaces, up to geometric equivalence, using Reeb graphs. Reeb graphs use
level-sets of functions, and the analogous join tree or merge tree structure
with sublevel sets was developed by \cite{pascucci2004} and \cite{Curry2017}.
Further implications of Reeb graphs for persistent homology were considered by~\cite{difabiolandi2016,bauerlandimemoli2018}.

From the dynamical system perspective, \cite{Peixoto1973} classified
Morse--Smale flows on two-manifolds up to trajectory topological equivalence
using the concept of ``distinguished graph". Subsequent work by
\cite{Fleitas1975} and  \cite{Wang1990} gave simpler invariants for Morse flows
on two-manifolds. \cite{OshemkovSharko1998} also considered the  problem of
topological trajectory classification of Morse--Smale flows on closed surfaces
and introduced a ``three-color graph" as another alternative to the Peixoto
invariant. Morse flows are also used to determine two-dimensional Hamiltonian flows; and \cite{SakajoYokoyama2018} developed tree representations for such flows. Both three-color graphs and tree representations are combinatorial codings that detail processes for constructing a flow by adding pairs of critical points. Finally, \cite{AdamsCarlsson2015} used topological arguments like the ones in Section~\ref{sec:results-posets} to decompose spaces for network evasion paths, however they did not use the language of Morse functions. Our work focuses on invariants based on cell decompositions of the domain using gradient flows. It is primarily concerned with the structure of the Morse--Smale complex and its interplay with persistence, which is distinct from previous approaches.

The formal underpinnings of the sequence of moves described in
Section~\ref{sec:results-abcd} originate from Cerf theory (also known as pseudoisotopy
theory)~\cite{Cerf1970,hatcher_pseudo-isotopies_1973}. In proving his celebrated
``Pseudoisotopy Theorem'', Cerf described the low codimension strata of a particular stratification
on the space of smooth functions on a smooth compact manifold. The codimension-zero stratum consists of Morse functions with distinct critical values, and the codimension-one stratum consists of either `generalized Morse functions' (those with a single cubic `birth-death' singularity), or Morse functions with precisely two critical values equal. Furthermore, Cerf showed that a generic or typical path of smooth functions lies in the codimension-zero stratum for all but finitely many `times' (thinking of the path parameter as `time'), at which points the function lies in the codimension-one stratum. The moves described in Section~\ref{sec:results-abcd} are inspired by moving across the codimension-one stratum in the space of smooth functions and passing through a generalized Morse function.  The statement and proof of Theorem~\ref{thm:exh_moves} relies on this description of the stratification and adapts these ideas to the combinatorics of the Morse--Smale graph of a surface.

Performing one of the moves introduced in Section~\ref{sec:results-abcd} can be thought of as a perturbation of the initial Morse-Smale flow (albeit a fairly large one from the perspective the aforementioned stratification of Cerf). Perturbations in the space of all vector fields have been studied by other authors (see~\cite{szymczak2012hierarchy} and its extensive references). This perspective might be useful for future work, but we do not take this approach now.

\subsection{Overview}
Our main contributions in this paper are:
\begin{itemize}
\item New notions of equivalence among (embeddings of) Morse functions containing persistence information;
\item A set of fundamental operations on Morse--Smale vector fields that relate all such vector fields;
\item A foundation for counting the number of Morse functions producing the same barcode.
\end{itemize}

In Section \ref{sec:morsefuncs-persistence} we begin with a background to the functions of interest in the context of persistent homology, with new and existing notions of equivalence among these functions in Section \ref{sec:morsefunc-equivalence}. Section \ref{sec:results-abcd} contains a method for relating Morse--Smale vector fields on the sphere, building from existing decomposition results \cite{EdelsbrunnerHarerZomorodian2003}. Section \ref{sec:results-posets} explores one of the new invariants, the nesting poset, and describes a zigzag poset structure in Corollary \ref{cor:zigzag}, making steps to extend it combinatorially in Conjecture \ref{conj:zigzag-algebra}, with a goal of developing an enriched barcode. Section \ref{sec:results-counting} presents a lower bound in Conjecture \ref{conj:counting} to counting height-equivalence classes of Morse functions that factor through $\Rspace^3$ as smooth embeddings, which leads to better understanding of Morse functions by their barcode.

\section{Technical Background}
\label{sec:background}

We first summarize relevant aspects of Morse theory,  see \cite{Milnor1963}, \cite{Matsumoto1997}, and \cite{Nicolaescu2007} for detailed expositions on the topic. We then review known notions of equivalence among Morse functions. After that, we introduce the notions of graph equivalence, height equivalence, and poset equivalence. We conclude by a comparison of equivalence relations among Morse functions.

\subsection{Morse Functions and Persistence}
\label{sec:morsefuncs-persistence}

Let $\Mspace$ be a smooth, compact, orientable manifold, equipped with a Riemannian metric $g_{\Mspace}$, and $f \colon \Mspace \to \Rspace$ be a smooth function.

\subsubsection{Morse Functions}
A critical point $p$ of $\Mspace$ is \emph{non-degenerate} if there exists a chart $(x_1,\ldots, x_n)$ on a neighborhood $U$ of $p$ such that
\begin{enumerate}
\item $x_i(p) = 0$ for all $i$, and
\item $f(x) = f(p) - x_1^2 -\cdots - x_\lambda^2 + x_{\lambda+1}^2 +\cdots + x_n^2$.
\end{enumerate}
The number $\lambda$ is the \emph{(Morse) index} of the critical point $p$, and is independent of the choice of chart. The index of a critical point is an integer between $0$ and the dimension of $\Mspace$. A smooth function $f\colon \Mspace \to \Rspace$ is a \emph{Morse function} if all its critical points are non-degenerate.  Furthermore, $f$ is an \emph{excellent Morse function} if all critical points have distinct function values~\cite{Nicolaescu2008}. All Morse functions considered in this paper are excellent, and referred to simply as Morse functions.

\subsubsection{Handle Decomposition}
For $f \colon \Mspace \to \Rspace$ a Morse function, let $\Mspace_t \colonequals f^{-1}(-\infty, t] = \{ x \in \Mspace \mid f(x) \leq t\}$ denote sublevel sets of $f$. Morse theory studies how $\Mspace_t$ changes as the parameter $t$ changes. There are two fundamental theorems regarding handle decomposition of manifolds in Morse theory~\cite{Milnor1963,Matsumoto1997,Nicolaescu2007}.
\begin{theorem}[{\cite[Theorem 3.1]{Milnor1963}}]
\label{theorem:CMT-A}
If $f$ has no critical values in the real interval $[a,b]$, then $\Mspace_a$ and $\Mspace_b$ are diffeomorphic.
\end{theorem}

\begin{theorem}[{\cite[Theorem 3.2]{Milnor1963},\cite[page 77]{Matsumoto1997}}]
\label{theorem:CMT-B}
\edits{Let $p$ be a critical point of index $\lambda$ with critical value $c = f(p)$.  Suppose that for some $\varepsilon>0$, the set $f^{-1}([p-\varepsilon, p+\varepsilon])$ contains no critical points of $f$ besides $p$. Then the space $\Mspace_{c+\epsilon}$ is diffeomorphic to the manifold obtained by attaching a $\lambda$-handle to $\Mspace_{c-\epsilon}$. That is, $\Mspace_{c+\epsilon}$ is diffeomorphic to $\Mspace_{c-\epsilon} \cup \Dspace^{\lambda} \times \Dspace^{d-\lambda}$, where $\Dspace^\lambda$ denotes a $\lambda$-dimensional disk.}
\end{theorem}

Summarizing, the sublevel sets of a Morse function change precisely when passing through a critical value. Moreover, this change is completely characterized topologically by the index of the critical point.

\subsubsection{Gradient Vector Fields}
A \emph{vector field} on a manifold is a smooth section of the tangent bundle. Equivalently, it is a smooth function $v: \Mspace \to \TM$, such that $v(x) \in \TM_{x}$, where $\TM_x$ is the tangent space of $\Mspace$ at $x$. Given a smooth function $f\colon \Mspace \to \Rspace$, the \emph{gradient} of $f$ (with respect to the metric $g_\Mspace$) is a vector field $\grad{f}\colon \Mspace \to \TM$ consisting of vectors in the direction of the steepest ascent of $f$, and is formally dual to the differential $df$. The singularities of $\grad{f}$ coincide with the critical points of $f$, and hence are isolated and finite.

\subsubsection{Morse--Smale Functions}
Let $\phi_t$ denote the flow generated by $\grad{f}$. For a critical point $p \in \Mspace$ of $f$, the \emph{stable manifold} of $p$ is $S(p) = \{x \in \Mspace \mid \lim_{t \to \infty} \phi_t(x) = p\} \,. $ The \emph{unstable manifold} of $p$ is $U(p) = \{x \in \Mspace \mid \lim_{t \to -\infty} \phi_t(x) = p\} \,. $ A \emph{Morse--Smale function} is a Morse function whose stable and unstable manifolds intersect transversally. The Morse--Smale condition is dependent on the metric $g_{\Mspace}$, but we omit this from the terminology.

An \emph{integral curve} of $f$ passing through a regular point $x$ is $\gamma = \gamma_x \colon \Rspace \to \Mspace$ defined by $\gamma(t) = \phi_t(x)$. A \emph{flow line} is an equivalence class of integral curves of $f$, where $\gamma\sim \gamma'$ if $\gamma(t) = \gamma'(s+t)$ for some $s$ and all $t\in \Rspace$. Therefore the unstable and stable manifolds of a critical point are the unions of all flow lines which begin and terminate, respectively, at that critical point. The Morse--Smale condition imposes restrictions on flow lines. For example, flow lines of a Morse--Smale gradient cannot connect critical points of the same index.

For a given Morse--Smale function $f$, by intersecting the stable and unstable
manifolds, we obtain the \emph{Morse--Smale cells} as the connected components
of the set $U(p) \cap S(q)$ for all critical points $p, q \in
\Mspace$~\cite{EdelsbrunnerHarerZomorodian2003}.  The \emph{Morse--Smale
complex} is the collection of Morse--Smale
cells~\cite{EdelsbrunnerHarerZomorodian2003}.\footnote{The Morse--Smale complex
described here is treated as a combinatorial structure, not to be confused with
Morse--Smale--Witten chain complex~\cite[Chapter 7]{BanyagaHurtubise2013}.} We
define the \emph{Morse--Smale graph} of $f$ to be the 1-skeleton of the
Morse--Smale complex, that is, the union of the zero-dimensional (vertices) and
one-dimensional (edges) cells of the Morse--Smale complex of $f$.

\begin{remark}
The Morse--Smale graph is also referred to as the \emph{topological skeleton} in visualization~\cite{HelmanHesselink1989}, consisting of critical points and streamlines that connect them which divide the domain of $\grad{f}$ into areas of different flow behavior (referred to as separatrices).  A similar invariant is the \emph{distinguished graph} of a gradient-like flow~\cite{Peixoto1973}, allowing for the possibility of maximum-minimum connections, which never occur in the Morse--Smale graph~\cite[Quadrangle Lemma]{EdelsbrunnerHarerZomorodian2003}.
\end{remark}

Let $V_f$ denote the vertices of the Morse--Smale graph. We define the \emph{decorated Morse--Smale graph} of a Morse function $f\colon \Mspace \to \Rspace$ to be the Morse--Smale graph of~$f$ equipped with a vertex weighting given by restricting $f$ to the vertices: $f|_{V_f}\colon V_f \to \Rspace$. Figure~\ref{fig:msc-example} is an example of a decorated Morse--Smale graph with the vertex weighting marked next to the critical points. We begin with a Morse function on the sphere $f\colon \Sspace^2 \to \Mspace$ with three maxima, three saddles and two minima.  We imagine cutting open and replacing the global minimum with an elastic band, and mapping the sphere to a disk for a clearer visualization.

\begin{figure}[h]\centering
\begin{tikzpicture}
\draw[blue,fill=cellfill] (0,0) circle (2);
\foreach \x\y\nam in {-1.5/0/a, -.2/0/b, 1.5/0/c}{
  \coordinate (\nam) at (\x,\y);
}
\coordinate (ab) at ($(a)!.5!(b)$);
\coordinate (bc) at ($(b)!.5!(c)$);
\coordinate (bct) at ($(bc)+(90:.8)$);
\coordinate (bcb) at ($(bc)+(270:.8)$);
\draw[densely dashed] (a)--(b);
\draw[densely dashed] (b) to [out=90,in=180] (bct);
\draw[densely dashed] (b) to [out=270,in=180] (bcb);
\draw[densely dashed] (bct) to [out=0,in=90] (c);
\draw[densely dashed] (bcb) to [out=0,in=270] (c);
\draw (ab) to [out=90,in=315] ($(0,0)+(135:2)$);
\draw (ab) to [out=270,in=45] ($(0,0)+(225:2)$);
\draw (bct)--(bcb);
\draw (bct) to [out=90,in=240] ($(0,0)+(60:2)$);
\draw (bcb) to [out=270,in=120] ($(0,0)+(-60:2)$);
\foreach \coord\typ in {a/\maxx, b/\maxx, c/\maxx, ab/\sadd, bc/\minn, bct/\sadd, bcb/\sadd}{
  \typ{\coord}
}
\node[color=blue] at ($(0,0)+(225:2.3)$) {1};
\node[anchor=north west,color=blue] at (bc.south west) {2};
\node[anchor=45,color=dgren,inner sep=5pt] at (bcb.south west) {3};
\node[anchor=315,color=dgren,inner sep=5pt] at (bct.north west) {4};
\node[anchor=120,color=dgren,inner sep=5pt] at (ab.south west) {5};
\node[anchor=west,color=red,inner sep=5pt] at (c.east) {6};
\node[anchor=300,color=red,inner sep=5pt] at (b.north east) {7};
\node[anchor=east,color=red,inner sep=5pt] at (a.west) {8};
\foreach \y in {1,...,6}{
  \coordinate (\y) at (4,2.3-\y*.7);
}
\maxx{1}
\sadd{2}
\minn{3}
\draw[densely dashed] (4)--+(180:.75);
\draw (5)--+(180:.75);
\draw[blue] (6)--+(180:.75);
\node[anchor=182] at (1) {\ maximum};
\node[anchor=182] at (2) {\ saddle};
\node[anchor=182] at (3) {\ minimum};
\node[anchor=181] at (4) {\ saddle-max connection};
\node[anchor=181] at (5) {\ saddle-min connection};
\node[anchor=west] at (6) {\ global minimum};
\end{tikzpicture}
\caption{A decorated Morse--Smale graph for a Morse function on the sphere. The boundary of this disk is identified to a point, which is the global minimum with weight 1.}
\label{fig:msc-example}
\end{figure}
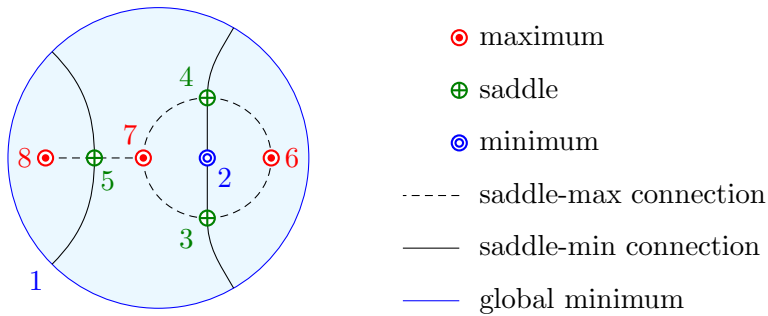

\subsubsection{Filtrations for Persistent Homology}

In this paper, we are mostly concerned with \emph{sublevel set} filtrations of
functions. That is, we are interested in the topological and algebraic
properties of sets $f^{-1}(-\infty,t]$ for $t\in \Rspace$, and inclusion maps
among them. Section \ref{sec:results-posets} is an exception, where \emph{level
set} and \emph{interlevel set} filtrations are considered, that is, we use sets
of the sort $f^{-1}(t)$ and $f^{-1}[t-\epsilon,t+\epsilon]$ for $t\in \Rspace$
and $\epsilon>0$. We refer the reader to broader surveys such
as~\cite{EdelsbrunnerHarer2010,BendichEdelsbrunnerMorozov2013} for more on the
different ways to approach persistence.

\subsection{Equivalences Among Morse Functions}
\label{sec:morsefunc-equivalence}

We first review several equivalence relations between Morse functions that have been studied in the literature, including \emph{geometric equivalence}, \emph{topological equivalence}, and \emph{homological equivalence}. We then introduce new notions of equivalence relations between Morse--Smale functions that are essential to our research objectives, namely, \emph{graph equivalence}, \emph{height equivalence}, and \emph{poset equivalence}.

\subsubsection{Orientation Preservation and Level Set Preservation}
Let $\Mspace$ and $\Nspace$ be smooth, oriented manifolds (of dimension $n$). A diffeomorphism $h\colon \Mspace \to \Nspace$ is \emph{orientation-preserving} provided that $dh_p$ preserves the orientation at each point $p$ of $\Mspace$, that is, the linear transformation $dh_p$ has positive determinant.

Given two Morse functions on manifolds, $f\colon \Mspace \to \Rspace$ and $g\colon \Nspace \to \Rspace$, a homeomorphism $h\colon \Mspace \to \Nspace$ is \emph{level set preserving} if $h(f^{-1}(a)) = g^{-1}(a)$ for any $a \in \Rspace$. Equivalently, $h\colon \Mspace \to \Nspace$ is level-set preserving if and only if the diagram
\begin{equation}
\begin{tikzcd}
 \Mspace \arrow[rr, rightarrow, "h"] \arrow[dr, rightarrow, "f"'] & & \Nspace \arrow[dl, rightarrow, "g"] \\
& \Rspace &
\end{tikzcd}
\end{equation}
commutes.

\subsubsection{Geometric, Topological and Homological Equivalences}
As before, let $f\colon \Mspace \to \Rspace$ be Morse and  $\Mspace_{t}^f \colonequals f^{-1}(\infty, t]$ its sublevel sets. For $n_f$ the number of critical points of $f$, let $a_0 < \cdots < a_{n_f}$ be a sequence of regular values of $f$ such that each interval $(a_i, a_{i+1})$ contains exactly one critical value of $f$ (for $0 \leq i \leq n_f-1$), called a \emph{slicing}~\cite{Nicolaescu2008} of $f$. Two Morse functions $f, g\colon \Mspace \to \Rspace$ are \emph{geometrically equivalent} if there exist orientation-preserving diffeomorphisms $r\colon \Mspace \to \Mspace$ and $l\colon \Rspace \to \Rspace$ such that $g = l \circ f \circ r^{-1}$, or equivalently, if the diagram
\begin{equation}\label{eqn:geom_equiv}
\begin{tikzcd}
 \Mspace
 \arrow[r,rightarrow,"r"]
 \arrow[d,"f"]
 &
 \Mspace
 \arrow[d,"g"]
 \\
 \Rspace
 \arrow[r,rightarrow,"l"]
 &
 \Rspace
\end{tikzcd}\end{equation}
commutes. The Morse functions $f$ and $g$ are \emph{topologically equivalent} if they have the same number of critical values $n_f = n_g$ and there \edits{exists} a slicing $a_0 < \cdots < a_{n_f}$ of $f$ and a slicing $b_0 < \cdots < b_{n_g}$ of $g$ together with orientation-preserving diffeomorphisms $\phi_i\colon \Mspace_{a_i}^f \to \Mspace_{b_i}^g$ between sublevel sets. They are \emph{(mod $p$) homologically equivalent} if they have the same number of critical points and there exists a slicing of $f$ and a slicing of $g$ such that each of the sublevel sets  $\Mspace_{a_i}^f$ and~$\Mspace_{b_i}^g$ have the same (mod $p$) Betti numbers. Note that geometric equivalence implies topological equivalence, which in turn, implies homological equivalence; see~\cite{Nicolaescu2008} for details.

\subsubsection{Graph Equivalence}
Two Morse--Smale functions $f,g$ are \emph{graph equivalent} if there is a graph isomorphism $\varphi\colon V_f \to V_g$ with $f|_{V_f} = g|_{V_g}\circ \varphi$. Graph equivalence is strictly stronger than topological equivalence and level-set equivalence, as described in Figure~\ref{fig:graph-non-equivalence}.

\begin{figure}[h]\centering
\newcommand\vfactor{2.5} 
\begin{tikzpicture}[scale=1.2]
\draw[blue,fill=cellfill] (0,0) circle (1.5);
\foreach \r\n in {90/a, 210/b, 330/c}{
  \coordinate (\n) at ($(0,0)+(\r:1)$);
  \draw (0,0)--++(\r+60:1.5);
}
\draw[densely dashed] (a)--(b)--(c)--(a);
\foreach \n\m\K in {a/b/ab, b/c/bc, c/a/ca}{
  \coordinate (\K) at ($(\n)!.5!(\m)$);
}
\foreach \coord\typ in {a/\maxx, b/\maxx, c/\maxx, ab/\sadd, bc/\sadd, ca/\sadd, {0,0}/\minn}{
  \typ{\coord}
}
\node[color=blue] at ($(0,0)+(225:1.8)$) {1};
\node[anchor=north west,color=blue] at (0,0) {2};
\node[anchor=285,color=dgren,inner sep=7pt] at (ab.north) {3};
\node[anchor=255,color=dgren,inner sep=7pt] at (ca.north) {4};
\node[anchor=north west,color=dgren] at (bc.south east) {5};
\node[anchor=east,color=red,inner sep=5pt] at (a.west) {6};
\node[anchor=north,color=red,inner sep=7pt] at (c.south) {7};
\node[anchor=north,color=red,inner sep=7pt] at (b.south) {8};
\begin{scope}[shift={(4,0)}]
\draw[blue,fill=cellfill] (0,0) circle (1.5);
\foreach \x\y\nam in {-1/0/a, 0/0/b, 1/0/c}{
  \coordinate (\nam) at (\x,\y);
}
\coordinate (ab) at ($(a)!.5!(b)$);
\coordinate (bc) at ($(b)!.5!(c)$);
\coordinate (bct) at ($(bc)+(90:.5)$);
\coordinate (bcb) at ($(bc)+(270:.5)$);
\draw[densely dashed] (a)--(b);
\draw[densely dashed] (b) to [out=90,in=180] (bct);
\draw[densely dashed] (b) to [out=270,in=180] (bcb);
\draw[densely dashed] (bct) to [out=0,in=90] (c);
\draw[densely dashed] (bcb) to [out=0,in=270] (c);
\draw (ab) to [out=90,in=315] ($(0,0)+(135:1.5)$);
\draw (ab) to [out=270,in=45] ($(0,0)+(225:1.5)$);
\draw (bct)--(bcb);
\draw (bct) to [out=90,in=240] ($(0,0)+(60:1.5)$);
\draw (bcb) to [out=270,in=120] ($(0,0)+(-60:1.5)$);
\foreach \coord\typ in {a/\maxx, b/\maxx, c/\maxx, ab/\sadd, bc/\minn, bct/\sadd, bcb/\sadd}{
  \typ{\coord}
}
\node[color=blue] at ($(0,0)+(225:1.8)$) {1};
\node[anchor=north west,color=blue] at (bc.south west) {2};
\node[anchor=45,color=dgren,inner sep=5pt] at (bcb.south west) {3};
\node[anchor=315,color=dgren,inner sep=5pt] at (bct.north west) {4};
\node[anchor=120,color=dgren,inner sep=5pt] at (ab.south west) {5};
\node[anchor=west,color=red,inner sep=5pt] at (c.east) {6};
\node[anchor=300,color=red,inner sep=5pt] at (b.north east) {7};
\node[anchor=east,color=red,inner sep=5pt] at (a.west) {8};
\end{scope}
\begin{scope}[shift={(9,-1.5)}]
\begin{scope}[shift={(-.4,0)}]
\draw[|->] (2.5,0)--(2.5,9/\vfactor);
\node[scale=.8] (r) at (2.5,-.5) {$\Rspace$\vphantom{$H_0$}};
\foreach \y\col in {1/blue, 2/blue, 3/dgren, 4/dgren, 5/dgren, 6/red, 7/red, 8/red}{
  \draw[thick,draw=\col] (2.4,\y/\vfactor) to (2.6,\y/\vfactor) node[scale=.8,fill=white,text=\col,right=0pt] {$\y$};
  \draw[opacity=.5,dotted] (-2.2,\y/\vfactor)--(2.2,\y/\vfactor);
}
\end{scope}
\begin{scope}[shift={(-1.1,0)}]
\node[scale=.8] (rg) at (-1,-.3) {Reeb\vphantom{Ag}};
\node[scale=.8,yshift=5pt,anchor=north] at (rg.south) {graph};
\draw[line width=1pt] (-1.2,1/\vfactor) .. controls +(90:.2) and +(180:.2) .. (-1,3/\vfactor);
\draw[line width=1pt] (-.8,2/\vfactor) .. controls +(90:.2) and +(0:.2) .. (-1,3/\vfactor);
\draw[line width=1pt] (-1,3/\vfactor) -- (-1,4/\vfactor);
\draw[line width=1pt] (-1,4/\vfactor) .. controls +(0:.2) and +(270:.2) .. (-.8,6/\vfactor);
\draw[line width=1pt] (-1,4/\vfactor) .. controls +(180:.2) and +(270:.2) .. (-1.2,5/\vfactor);
\draw[line width=1pt] (-1.2,5/\vfactor) .. controls +(180:.2) and +(270:.2) .. (-1.4,7/\vfactor);
\draw[line width=1pt] (-1.2,5/\vfactor) .. controls +(0:.2) and +(270:.7) .. (-1,8/\vfactor);
\end{scope}
\node[scale=.8] (mt) at (-1,-.3) {merge\vphantom{Ag}};
\node[scale=.8,yshift=5pt,anchor=north] at (mt.south) {tree};
\draw[line width=1pt] (-1.2,1/\vfactor) .. controls +(90:.2) and +(180:.2) .. (-1,3/\vfactor);
\draw[line width=1pt] (-.8,2/\vfactor) .. controls +(90:.2) and +(0:.2) .. (-1,3/\vfactor);
\draw[line width=1pt] (-1,3/\vfactor) -- (-1,8/\vfactor);
\begin{scope}[shift={(-.5,0)}]
\node[scale=.8] (h0) at (.65,-.3) {$H_0$};
\node[scale=.8,yshift=5pt,anchor=north] at (h0.south) {barcode};
\draw[barcinf] (.5,1/\vfactor)--(.5,8/\vfactor);
\draw[barco] (.8,2/\vfactor)--(.8,3/\vfactor);
\end{scope}
\begin{scope}[shift={(-1,0)}]
\node[scale=.8](h1) at (2.3,-.3) {$H_1$};
\node[scale=.8,yshift=5pt,anchor=north] at (h1.south) {barcode};
\draw[barcc] (2,1/\vfactor)--(2,8/\vfactor);
\draw[barco] (2.3,4/\vfactor)--(2.3,6/\vfactor);
\draw[barco] (2.6,5/\vfactor)--(2.6,7/\vfactor);
\end{scope}
\end{scope}
\end{tikzpicture}
\caption{An example of two Morse functions $\Sspace^2\to \Rspace$ that have the same barcode, Reeb graph, and merge tree. These functions are geometrically equivalent by~\cite[Theorem 3.3]{Nicolaescu2008}, but are not graph equivalent.}
\label{fig:graph-non-equivalence}
\end{figure}

\subsubsection{Height Equivalence}
Let $\iota, \iota'\colon \Sspace^2 \to \Rspace^3$ be smooth embeddings of a sphere to $\Rspace^3$. Let~$\pi\colon \Rspace^3 \to \Rspace$ be a projection onto the unit normal vector $[0,0,1]^T$.  Let $f, g\colon \Sspace^2 \to \Rspace$ be two Morse functions that factor through the two embeddings $\iota$ and $\iota'$, respectively; that is,~$f = \pi \circ \iota$ and $g = \pi \circ \iota'$. Then $f$ and $g$ are called \emph{height equivalent} whenever there is a level set preserving homeomorphism~$\psi\colon \Rspace^3 \to \Rspace^3$ with $\iota'=\psi \circ \iota$. Equivalently,  $f$ and $g$ are height equivalent if the diagram
\begin{equation}
\label{eqn:height-equiv-diag}
\begin{tikzcd}
 & \Sspace^2 \arrow[dl,hookrightarrow,"\iota"'] \arrow[dr, hookrightarrow, "\iota' "] & \\
 \Rspace^3 \arrow[rr, rightarrow, "\psi"] \arrow[dr, twoheadrightarrow, "\pi"'] & & \Rspace^3 \arrow[dl, twoheadrightarrow, "\pi"] \\
& \R &
\end{tikzcd}
\end{equation}
commutes. Two height equivalent Morse functions are necessarily equal as functions $\Sspace^2 \to \R$, and thus will have the same critical values and (sub)level set persistence barcodes.

\subsubsection{Poset Equivalence}
We remind the reader that an \emph{isomorphism} from a poset  $(F,\leq _{F})$ to a poset $(G,\leq _{G})$ is a bijective function $\varphi\colon F\to G$ of sets with the property that, for every $x$ and $y$ in~$F$, $x\leq _{F}y$ if and only if $\varphi(x)\leq _{G}\varphi(y)$.

\begin{definition}[Poset Equivalence]
\label{def:poset-equivalence}
Two Morse functions $f, g\colon \Sspace^2 \to \Rspace$ factoring through embeddings $\iota$, $\iota'$, respectively, are \textit{poset equivalent} if they are height equivalent, and if
\begin{enumerate}
\item there exists a common slicing $a_0 < \cdots < a_{n}$ of $f$ and $g$ such that, for every $i$, the sets $F_i \colonequals \pi_0\left(\pi^{-1}(a_i) - \iota \circ f^{-1}(a_i)\right)$ and $G_i \colonequals \pi_0\left(\pi^{-1}(a_i) - \iota' \circ g^{-1}(a_i)\right)$ have the structure of a poset, and
\item the map $\pi_0 \circ \psi_i \colon F_i \to G_i $ induced by $\psi_i \colon \left(\pi^{-1}(a_i) - \iota \circ f^{-1}(a_i)\right) \to \left(\pi^{-1}(a_i) - \iota' \circ g^{-1}(a_i)\right)$ is an isomorphism of posets, where $\psi_i$ is a restriction of $\psi$ to the planes $\pi^{-1}(a_i)$.
\end{enumerate}
\end{definition}

This equivalence is necessary for understanding the preimage of the persistence map in Section~\ref{sec:results-posets}. Both Figure~\ref{fig:nesting}(a) and Figure~\ref{fig:nesting}(c) give two examples of this poset structure, and Figure~\ref{fig:shotglassworm} describes these examples in the context of their Morse functions.

\subsubsection{Comparison of Equivalence Relations}
We conclude this section with a comparison of the equivalence relations introduced thus far, for Morse functions $f\colon \Sspace^2\to \Rspace$.
\begin{lemma}
The following are strict implications among the equivalence relations.
\begin{center}
\begin{tikzpicture}[yscale=1.5,stronger/.style={double,line width=.8pt,double distance=2pt,-Implies}]
\node (poset) at (0,0) {poset};
\node[anchor=west] (height) at ($(poset.east)+(0:1)$) {height};
\node[anchor=west] (geo) at ($(height.east)+(0:1)$) {geometric};
\node[anchor=west] (top) at ($(geo.east)+(0:1)$) {topological};
\node[anchor=west] (hom) at ($(top.east)+(0:1)$) {homological};
\node (gr) at ($(geo)+(270:1)$) {graph};
\foreach \x\y in {poset/height, height/geo, geo/top, top/hom}{
  \draw[stronger] (\x)--(\y);
}
\draw[stronger] (geo) to node[right=-4.5pt,rotate=-45,pos=.25,scale=1.2] {$|$} (gr);
\end{tikzpicture}
\end{center}
\end{lemma}

\begin{proof}
In the top line, the first implication follows directly by definition, and is strict by the example of Figure~\ref{fig:shotglassworm}. The second implication follows from taking $r = \id_{\Sspace^2}$ and $l = \id_{\Rspace}$ in Diagram~\eqref{eqn:geom_equiv}. Given any Morse function $f \colon \Sspace^2 \to \Rspace$, and sufficiently small $\epsilon$, the functions $f$ and $f+\epsilon$ are geometrically equivalent, but not height equivalent (because they  are distinct). The third and fourth implications are shown in~\cite{Nicolaescu2008}, and Figure 5 in the same paper shows that the third implication is strict. As topological equivalence considers orientation while homological equivalence does not, examples abound of why the fourth implication is strict. Figure~\ref{fig:graph-non-equivalence} gives an example of two Morse functions that are geometrically equivalent but not graph equivalent. \qed
\end{proof}

\edits{Both the cell decomposition in the decorated Morse--Smale graph and recent work \cite[Theorem 1]{morsebott} suggest that graph equivalence implies geometric equivalence. However, this remains an open problem.}

\begin{remark}
If two functions $f,g \colon \Sspace^2 \to \Rspace$ are homologically equivalent, this does not imply they have the same barcode. However, if we instead consider $\varepsilon$-interleavings~\cite{ChazalCohen-SteinerGlisse2009} of barcodes (thought of as persistence modules), then homological equivalence does imply an $\varepsilon$-interleaving.
\end{remark}

\section{Fundamental Moves}
\label{sec:results-abcd}

We focus on understanding how cells (generically as quadrangles) of a Morse--Smale complex fit together on a surface and how they change when a pair of critical points is added or removed. We only consider Morse--Smale complexes that arise from a Morse-Smale function on a sphere, and we refer to the changes as fundamental moves, or \emph{moves} in short. Our first main result is to define moves on the Morse--Smale complex, with the goal of describing all the possible ways to create a new Morse--Smale function.

By the Quadrangle Lemma \cite{EdelsbrunnerHarerZomorodian2003}, every face (cell) of a decorated Morse--Smale graph has four edges, counting an edge twice if the face is on both sides of the edge. This allows us to describe changes to the graph as a composition of moves. 

The gradient of a Morse--Smale function gives rise to a Morse--Smale vector field, therefore our approach equivalently describes changes to a Morse--Smale vector field due to the moves, with the changes limited to a particular region of cells for each move. Everything in the vector field outside of this region stays the same between moves. \edits{As we only investigate Morse--Smale functions on a manifold, by definition, all saddles are simple}; that is, every saddle has degree four, and the endpoints of the four adjacent edges alternate between maxima and minima. All higher-order saddles can be unfolded into simple saddles. As in Figure \ref{fig:msc-example}, a saddle-maximum connection is indicated by a solid line, and a saddle-minimum connection is marked by a dashed line. Maxima and minima may have arbitrary degrees. 

We now describe face moves, edge moves, and vertex moves; which operate on faces,  edges, and vertices, respectively. All of the moves add or remove two cells to the quadrangulation, or equivalently, they add or remove one saddle-maximum or saddle-minimum pair.  

\begin{definition}[Face Moves]\label{def:facemove}
A \emph{face move} is \edits{addition (cancellation)} of a pair of critical points in the interior of a cell.  
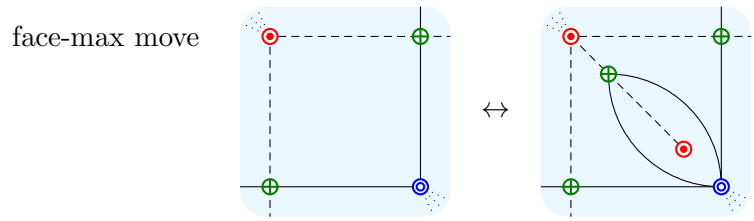
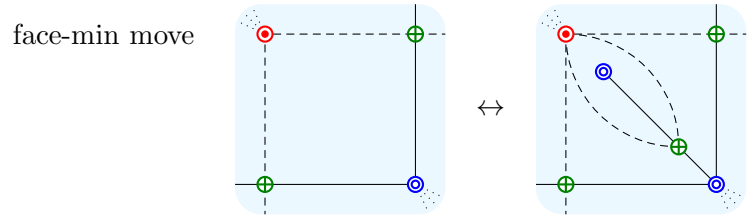
\[\begin{tikzpicture}
\begin{scope}
\foreach \x\y\name in {0/2/a, 2/2/b, 0/0/c, 2/0/d}{
  \coordinate (\name) at (\x,\y);
}
\fill[rounded corners=10pt,cellfill] ($(c)+(225:\spacer)$) rectangle ($(b)+(45:\spacer)$);
\clip[rounded corners=10pt] ($(c)+(225:\spacer)$) rectangle ($(b)+(45:\spacer)$);
\draw (b)--(d)--(c);
\draw[densely dashed] (c)--(a)--(b);
\draw (b)--++(90:\spacer) (c)--++(180:\spacer);
\draw[densely dashed] (b)--++(0:\spacer) (c)--++(270:\spacer);
\draw[dotted] (a)--++(115:\spacer) (a)--++(135:\spacer) (a)--++(155:\spacer);
\draw[dotted] (d)--++(295:\spacer) (d)--++(315:\spacer) (d)--++(335:\spacer);
\foreach \name\type in {a/\maxx, b/\sadd, c/\sadd, d/\minn}{
  \type{\name}
}
\end{scope}
\node at (3,1) {$\leftrightarrow$};
\node[anchor=east] at (-.8,2) {face-max move};
\begin{scope}[shift={(4,0)}]
\foreach \x\y\name in {0/2/a, 2/2/b, 0/0/c, 2/0/d, .5/1.5/e, 1.5/.5/f}{
  \coordinate (\name) at (\x,\y);
}
\fill[rounded corners=10pt,cellfill] ($(c)+(225:\spacer)$) rectangle ($(b)+(45:\spacer)$);
\clip[rounded corners=10pt] ($(c)+(225:\spacer)$) rectangle ($(b)+(45:\spacer)$);
\draw (b)--(d)--(c);
\draw[densely dashed] (c)--(a)--(b) (a)--(f);
\draw (e) to [bend left=45] (d);
\draw (e) to [bend right=45] (d);
\draw (b)--++(90:\spacer) (c)--++(180:\spacer);
\draw[densely dashed] (b)--++(0:\spacer) (c)--++(270:\spacer);
\draw[dotted] (a)--++(115:\spacer) (a)--++(135:\spacer) (a)--++(155:\spacer);
\draw[dotted] (d)--++(295:\spacer) (d)--++(315:\spacer) (d)--++(335:\spacer);
\foreach \name\type in {a/\maxx, b/\sadd, c/\sadd, d/\minn, e/\sadd, f/\maxx}{
  \type{\name}
}
\end{scope}
\begin{scope}[shift={(0,-4)}]
\begin{scope}
\foreach \x\y\name in {0/2/a, 2/2/b, 0/0/c, 2/0/d}{
  \coordinate (\name) at (\x,\y);
}
\fill[rounded corners=10pt,cellfill] ($(c)+(225:\spacer)$) rectangle ($(b)+(45:\spacer)$);
\clip[rounded corners=10pt] ($(c)+(225:\spacer)$) rectangle ($(b)+(45:\spacer)$);
\draw (b)--(d)--(c);
\draw[densely dashed] (c)--(a)--(b);
\draw (b)--++(90:\spacer) (c)--++(180:\spacer);
\draw[densely dashed] (b)--++(0:\spacer) (c)--++(270:\spacer);
\draw[dotted] (a)--++(115:\spacer) (a)--++(135:\spacer) (a)--++(155:\spacer);
\draw[dotted] (d)--++(295:\spacer) (d)--++(315:\spacer) (d)--++(335:\spacer);
\foreach \name\type in {a/\maxx, b/\sadd, c/\sadd, d/\minn}{
  \type{\name}
}
\end{scope}
\node at (3,1) {$\leftrightarrow$};
\node[anchor=east] at (-.8,2) {face-min move};
\end{scope}
\begin{scope}[shift={(4,-4)}]
\foreach \x\y\name in {0/2/a, 2/2/b, 0/0/c, 2/0/d, .5/1.5/e, 1.5/.5/f}{
  \coordinate (\name) at (\x,\y);
}
\fill[rounded corners=10pt,cellfill] ($(c)+(225:\spacer)$) rectangle ($(b)+(45:\spacer)$);
\clip[rounded corners=10pt] ($(c)+(225:\spacer)$) rectangle ($(b)+(45:\spacer)$);
\draw (b)--(d)--(c) (e)--(d);
\draw[densely dashed] (c)--(a)--(b);
\draw[densely dashed] (a) to [bend left=45] (f);
\draw[densely dashed] (a) to [bend right=45] (f);
\draw (b)--++(90:\spacer) (c)--++(180:\spacer);
\draw[densely dashed] (b)--++(0:\spacer) (c)--++(270:\spacer);
\draw[dotted] (a)--++(115:\spacer) (a)--++(135:\spacer) (a)--++(155:\spacer);
\draw[dotted] (d)--++(295:\spacer) (d)--++(315:\spacer) (d)--++(335:\spacer);
\foreach \name\type in {a/\maxx, b/\sadd, c/\sadd, d/\minn, e/\minn, f/\sadd}{
  \type{\name}
}
\end{scope}
\end{tikzpicture}\]
\end{definition}

\begin{definition}[Edge Moves]\label{def:edgemove}
An \emph{edge move} is \edits{addition (cancellation)} of a pair of critical points on the edge of a cell. 
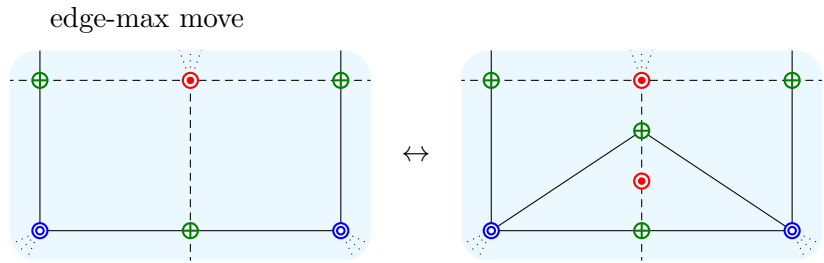
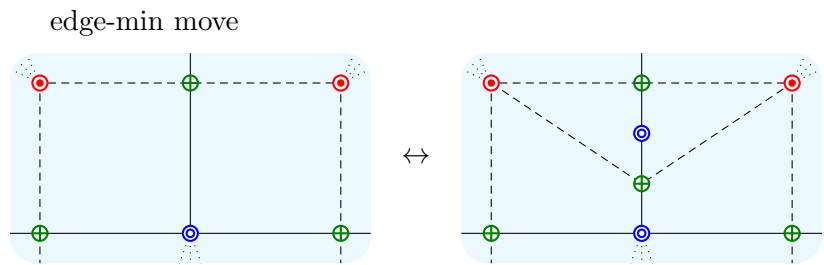
\[\begin{tikzpicture}
\begin{scope}
\foreach \x\y\name in {0/2/a, 2/2/b, 4/2/c, 0/0/d, 2/0/e, 4/0/f}{
  \coordinate (\name) at (\x,\y);
}
\fill[rounded corners=10pt,cellfill] ($(d)+(225:\spacer)$) rectangle ($(c)+(45:\spacer)$);
\clip[rounded corners=10pt] ($(d)+(225:\spacer)$) rectangle ($(c)+(45:\spacer)$);
\draw (a)--(d)--(f)--(c);
\draw[densely dashed] (a)--(c) (b)--(e);
\draw[densely dashed] (a)--++(180:\spacer) (c)--++(0:\spacer) (e)--++(270:\spacer);
\draw (a)--++(90:\spacer) (c)--++(90:\spacer);
\draw[dotted] (b)--++(70:\spacer) (b)--++(90:\spacer) (b)--++(110:\spacer);
\draw[dotted] (d)--++(205:\spacer) (d)--++(225:\spacer) (d)--++(245:\spacer);
\draw[dotted] (f)--++(295:\spacer) (f)--++(315:\spacer) (f)--++(335:\spacer);
\foreach \name\type in {a/\sadd, b/\maxx, c/\sadd, d/\minn, e/\sadd, f/\minn}{
  \type{\name}
}
\end{scope}
\node at (5,1) {$\leftrightarrow$};
\node[anchor=west] at (0,2.8) {edge-max move};
\begin{scope}[shift={(6,0)}]
\foreach \x\y\name in {0/2/a, 2/2/b, 4/2/c, 0/0/d, 2/0/e, 4/0/f, 2/1.33/top, 2/.66/bottom}{
  \coordinate (\name) at (\x,\y);
}
\fill[rounded corners=10pt,cellfill] ($(d)+(225:\spacer)$) rectangle ($(c)+(45:\spacer)$);
\clip[rounded corners=10pt] ($(d)+(225:\spacer)$) rectangle ($(c)+(45:\spacer)$);
\draw (a)--(d)--(f)--(c) (d)--(top)--(f);
\draw[densely dashed] (a)--(c) (b)--(e);
\draw[densely dashed] (a)--++(180:\spacer) (c)--++(0:\spacer) (e)--++(270:\spacer);
\draw (a)--++(90:\spacer) (c)--++(90:\spacer);
\draw[dotted] (b)--++(70:\spacer) (b)--++(90:\spacer) (b)--++(110:\spacer);
\draw[dotted] (d)--++(205:\spacer) (d)--++(225:\spacer) (d)--++(245:\spacer);
\draw[dotted] (f)--++(295:\spacer) (f)--++(315:\spacer) (f)--++(335:\spacer);
\foreach \name\type in {a/\sadd, b/\maxx, c/\sadd, d/\minn, e/\sadd, f/\minn, top/\sadd, bottom/\maxx}{
  \type{\name}
}
\end{scope}
\begin{scope}[shift={(0,-4)}]
\begin{scope}
\foreach \x\y\name in {0/2/a, 2/2/b, 4/2/c, 0/0/d, 2/0/e, 4/0/f}{
  \coordinate (\name) at (\x,\y);
}
\fill[rounded corners=10pt,cellfill] ($(d)+(225:\spacer)$) rectangle ($(c)+(45:\spacer)$);
\clip[rounded corners=10pt] ($(d)+(225:\spacer)$) rectangle ($(c)+(45:\spacer)$);
\draw (d)--(f) (e)--(b);
\draw[densely dashed] (d)--(a)--(c)--(f);
\draw[densely dashed] (d)--++(270:\spacer) (f)--++(270:\spacer);
\draw (b)--++(90:\spacer) (d)--++(180:\spacer) (f)--++(0:\spacer);
\draw[dotted] (e)--++(250:\spacer) (e)--++(270:\spacer) (e)--++(290:\spacer);
\draw[dotted] (a)--++(115:\spacer) (a)--++(135:\spacer) (a)--++(155:\spacer);
\draw[dotted] (c)--++(25:\spacer) (c)--++(45:\spacer) (c)--++(65:\spacer);
\foreach \name\type in {a/\maxx, b/\sadd, c/\maxx, d/\sadd, e/\minn, f/\sadd}{
  \type{\name}
}
\end{scope}
\node at (5,1) {$\leftrightarrow$};
\node[anchor=west] at (0,2.8) {edge-min move};
\end{scope}
\begin{scope}[shift={(6,-4)}]
\foreach \x\y\name in {0/2/a, 2/2/b, 4/2/c, 0/0/d, 2/0/e, 4/0/f, 2/1.33/top, 2/.66/bottom}{
  \coordinate (\name) at (\x,\y);
}
\fill[rounded corners=10pt,cellfill] ($(d)+(225:\spacer)$) rectangle ($(c)+(45:\spacer)$);
\clip[rounded corners=10pt] ($(d)+(225:\spacer)$) rectangle ($(c)+(45:\spacer)$);
\draw (d)--(f) (e)--(b);
\draw[densely dashed] (d)--(a)--(c)--(f) (a)--(bottom)--(c);
\draw[densely dashed] (d)--++(270:\spacer) (f)--++(270:\spacer);
\draw (b)--++(90:\spacer) (d)--++(180:\spacer) (f)--++(0:\spacer);
\draw[dotted] (e)--++(250:\spacer) (e)--++(270:\spacer) (e)--++(290:\spacer);
\draw[dotted] (a)--++(115:\spacer) (a)--++(135:\spacer) (a)--++(155:\spacer);
\draw[dotted] (c)--++(25:\spacer) (c)--++(45:\spacer) (c)--++(65:\spacer);
\foreach \name\type in {a/\maxx, b/\sadd, c/\maxx, d/\sadd, e/\minn, f/\sadd, top/\minn, bottom/\sadd}{
  \type{\name}
}
\end{scope}
\end{tikzpicture}\]
\end{definition}

\begin{definition}[Vertex Moves]\label{def:vertexmove}
A \emph{vertex move} is \edits{addition (cancellation)} of a pair of critical points at an existing critical point.
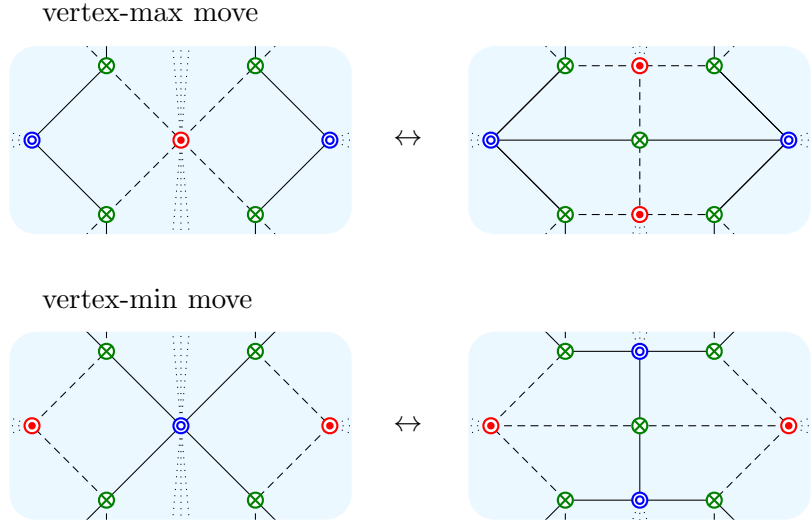
\[\begin{tikzpicture}
\fill[rounded corners=10pt,cellfill] (-.3,-1.25) rectangle (4.25,1.25);
\begin{scope}
\clip[rounded corners=10pt] (-.3,-1.25) rectangle (4.25,1.25);
\begin{scope}[rotate=45]
\foreach \x\y\name in {0/0/a, \elen/0/b, .75*\elen/0/bb, 0/-1*\elen/c, 0/-.75*\elen/cc, \elen/-1*\elen/d, 2*\elen/-1*\elen/e, 2*\elen/-1.25*\elen/ee, \elen/-2*\elen/f, 1.25*\elen/-2*\elen/ff, 2*\elen/-2*\elen/g}{
  \coordinate (\name) at (\x,\y);
}
\draw (c)--(a)--(b) (e)--(g)--(f);
\draw[densely dashed] (b)--(f) (c)--(e);
\draw[densely dashed] (b)--++(90:\spacer) (c)--++(180:\spacer) (e)--++(0:\spacer) (f)--++(270:\spacer);
\draw (b)--++(45:\twospacer) (c)--++(225:\twospacer) (e)--++(45:\twospacer) (f)--++(225:\twospacer);
\draw[dotted] (a)--++(115:\spacer) (a)--++(135:\spacer) (a)--++(155:\spacer);
\draw[dotted] (g)--++(295:\spacer) (g)--++(315:\spacer) (g)--++(335:\spacer);
\draw[dotted] (d)--++(40:1.5*\elen) (d)--++(45:1.5*\elen) (d)--++(50:1.5*\elen);
\draw[dotted] (d)--++(220:1.5*\elen) (d)--++(225:1.5*\elen) (d)--++(230:1.5*\elen);
\foreach \name\type in {a/\minn, b/\sadd, c/\sadd, d/\maxx, e/\sadd, f/\sadd, g/\minn}{
  \type{\name}
}
\end{scope}\end{scope}
\node at (5,0) {$\leftrightarrow$};
\node[anchor=west] at (0,1.7) {vertex-max move};
\begin{scope}[shift={(6.1,0)}]
\fill[rounded corners=10pt,cellfill] (-.3,-1.25) rectangle (4.25,1.25);
\begin{scope}
\clip[rounded corners=10pt] (-.3,-1.25) rectangle (4.25,1.25);
\begin{scope}[rotate=45]
\foreach \x\y\name in {0/0/a, 1*\elen/0/b, .75*\elen/0/bb, 0/-1*\elen/c, 0/-.75*\elen/cc, \elen/-1*\elen/d, 2*\elen/-1*\elen/e, 2*\elen/-1.25*\elen/ee, 1*\elen/-2*\elen/f, 1.25*\elen/-2*\elen/ff, 2*\elen/-2*\elen/g}{
  \coordinate (\name) at (\x,\y);
}
\coordinate (dt) at ($(b)!.5!(e)$);
\coordinate (db) at ($(c)!.5!(f)$);
\draw (c)--(a)--(b) (e)--(g)--(f);
\draw (c)--(a)--(b) (a)--(d) (e)--(g)--(f) (g)--(d);
\draw[densely dashed] (b)--(dt)--(e) (c)--(db)--(f) (dt)--(db);
\draw[densely dashed] (b)--++(90:\spacer) (c)--++(180:\spacer) (e)--++(0:\spacer) (f)--++(270:\spacer);
\draw (b)--++(45:\twospacer) (c)--++(225:\twospacer) (e)--++(45:\twospacer) (f)--++(225:\twospacer);
\draw[dotted] (a)--++(115:\spacer) (a)--++(135:\spacer) (a)--++(155:\spacer);
\draw[dotted] (g)--++(295:\spacer) (g)--++(315:\spacer) (g)--++(335:\spacer);
\draw[dotted] (dt)--++(25:1.5*\elen) (dt)--++(45:1.5*\elen) (dt)--++(65:1.5*\elen);
\draw[dotted] (db)--++(205:1.5*\elen) (db)--++(225:1.5*\elen) (db)--++(245:1.5*\elen);
\foreach \name\type in {a/\minn, b/\sadd, c/\sadd, d/\sadd, dt/\maxx, db/\maxx, e/\sadd, f/\sadd, g/\minn}{
  \type{\name}
}
\end{scope}\end{scope}\end{scope}
\begin{scope}[shift={(0,-3.8)}]
\fill[rounded corners=10pt,cellfill] (-.3,-1.25) rectangle (4.25,1.25);
\begin{scope}
\clip[rounded corners=10pt] (-.3,-1.25) rectangle (4.25,1.25);
\begin{scope}[rotate=45]
\foreach \x\y\name in {0/0/a, \elen/0/b, .75*\elen/0/bb, 0/-1*\elen/c, 0/-.75*\elen/cc, \elen/-1*\elen/d, 2*\elen/-1*\elen/e, 2*\elen/-1.25*\elen/ee, \elen/-2*\elen/f, 1.25*\elen/-2*\elen/ff, 2*\elen/-2*\elen/g}{
  \coordinate (\name) at (\x,\y);
}
\draw[densely dashed] (c)--(a)--(b) (e)--(g)--(f);
\draw (b)--(f) (c)--(e);
\draw (b)--++(90:\spacer) (c)--++(180:\spacer) (e)--++(0:\spacer) (f)--++(270:\spacer);
\draw[densely dashed] (b)--++(45:\twospacer) (c)--++(225:\twospacer) (e)--++(45:\twospacer) (f)--++(225:\twospacer);
\draw[dotted] (a)--++(115:\spacer) (a)--++(135:\spacer) (a)--++(155:\spacer);
\draw[dotted] (g)--++(295:\spacer) (g)--++(315:\spacer) (g)--++(335:\spacer);
\draw[dotted] (d)--++(40:1.5*\elen) (d)--++(45:1.5*\elen) (d)--++(50:1.5*\elen);
\draw[dotted] (d)--++(220:1.5*\elen) (d)--++(225:1.5*\elen) (d)--++(230:1.5*\elen);
\foreach \name\type in {a/\maxx, b/\sadd, c/\sadd, d/\minn, e/\sadd, f/\sadd, g/\maxx}{
  \type{\name}
}
\end{scope}\end{scope}
\node[anchor=west] at (0,1.7) {vertex-min move};
\node at (5,0) {$\leftrightarrow$};
\begin{scope}[shift={(6.1,0)}]
\fill[rounded corners=10pt,cellfill] (-.3,-1.25) rectangle (4.25,1.25);
\begin{scope}
\clip[rounded corners=10pt] (-.3,-1.25) rectangle (4.25,1.25);
\begin{scope}[rotate=45]
\foreach \x\y\name in {0/0/a, 1*\elen/0/b, .75*\elen/0/bb, 0/-1*\elen/c, 0/-.75*\elen/cc, \elen/-1*\elen/d, 2*\elen/-1*\elen/e, 2*\elen/-1.25*\elen/ee, 1*\elen/-2*\elen/f, 1.25*\elen/-2*\elen/ff, 2*\elen/-2*\elen/g}{
  \coordinate (\name) at (\x,\y);
}
\coordinate (dt) at ($(b)!.5!(e)$);
\coordinate (db) at ($(c)!.5!(f)$);
\draw[densely dashed] (c)--(a)--(b) (a)--(d) (e)--(g)--(f) (g)--(d);
\draw (b)--(dt)--(e) (c)--(db)--(f) (dt)--(db);
\draw (b)--++(90:\spacer) (c)--++(180:\spacer) (e)--++(0:\spacer) (f)--++(270:\spacer);
\draw[densely dashed] (b)--++(45:\twospacer) (c)--++(225:\twospacer) (e)--++(45:\twospacer) (f)--++(225:\twospacer);
\draw[dotted] (a)--++(115:\spacer) (a)--++(135:\spacer) (a)--++(155:\spacer);
\draw[dotted] (g)--++(295:\spacer) (g)--++(315:\spacer) (g)--++(335:\spacer);
\draw[dotted] (dt)--++(25:1.5*\elen) (dt)--++(45:1.5*\elen) (dt)--++(65:1.5*\elen);
\draw[dotted] (db)--++(205:1.5*\elen) (db)--++(225:1.5*\elen) (db)--++(245:1.5*\elen);
\foreach \name\type in {a/\maxx, b/\sadd, c/\sadd, d/\sadd, dt/\minn, db/\minn, e/\sadd, f/\sadd, g/\maxx}{
  \type{\name}
}
\end{scope}\end{scope}\end{scope}\end{scope}
\end{tikzpicture}\]
\end{definition}

The face, edge, and vertex moves are ways of manipulating the Morse--Smale complex to obtain another Morse--Smale complex. These moves do not have functional values associated with the critical points, so they are manipulations of the Morse--Smale complex and not of the underlying function.

\begin{theorem}
\label{thm:exh_moves}
The Morse--Smale graph of any two Morse--Smale functions is related by a sequence of face, edge, and vertex moves.
\end{theorem}

The proof of this theorem relies on the following restriction imposed on the Morse--Smale graph. We recall this fact from~\cite{EdelsbrunnerHarerZomorodian2003} without proof.

\begin{lemma}[Quadrangle Lemma]
Each region of the Morse--Smale complex is a quadrangle with vertices of index 0, 1, 2, and 1, in this order around the region. The boundary is possibly glued to itself along vertices and arcs.
\end{lemma}

\begin{proof}[Proof of Theorem~\ref{thm:exh_moves}]
We recall the following fact due to \cite{Cerf1970}: any two Morse functions on a manifold $\Mspace$, and therefore any two Morse flows, can be connected by a path in the space of all smooth functions on $\Mspace$. Furthermore, this path can be chosen to be comprised of Morse functions for all but finitely many times -- at which times the function has a single cubic degenerate critical point -- in addition to non-degenerate critical points. As one moves along this path of functions, the cubic degenerate critical point either gives rise or removes two non-degenerate critical points. Translating this general result to the Morse--Smale complex of a surface, it suffices to consider introducing or removing a pair of critical points into a quadrangle decomposition of the surface. Depending on where we introduce this pair, we shall obtain the different moves of Definitions~\ref{def:facemove}--\ref{def:vertexmove}. Notice that since the cubic degenerate critical point must give rise to critical points of adjacent index, a saddle point must always be involved in these moves. Since adding and removing critical points are symmetric operations, we only discuss the case of adding critical points.

First suppose the pair of critical points is added in the interior of a quadrangle. Then, as in Figure~\ref{fig:face-move-proof}, the introduced saddle must have two flow lines to a single vertex on the boundary of the quadrangle. If a maximum-saddle pair is added, the additional saddle must have two flow lines to the unique minimum, and for an added minimum-saddle pair, the saddle must have two flow lines to the unique maximum. This completely determines the move, and hence is the face-max or face-min move.
    
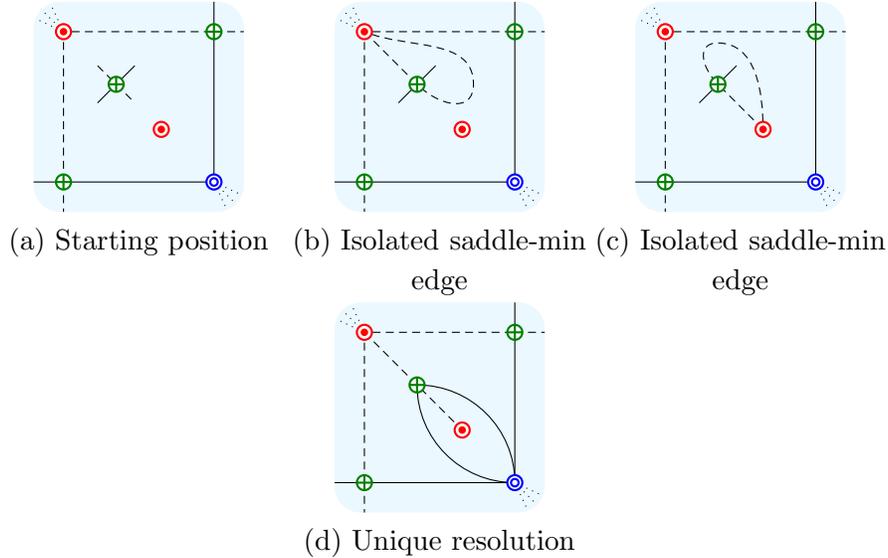
\begin{figure}[h]\centering
\newcommand\felen{2}
\newcommand\estretch{.35}
\newcommand\vfactor{4}
\begin{tikzpicture}[lab2/.style={anchor=north,yshift=3pt}]
\begin{scope}
\node (lab) at (\felen/2,-.8) {(a) Starting position};
\foreach \x\y\name in {0/\felen/a, \felen/\felen/b, 0/0/c, \felen/0/d, .35*\felen/.65*\felen/n1, .65*\felen/.35*\felen/n2}{
  \coordinate (\name) at (\x,\y);
}
\fill[rounded corners=10pt,cellfill] ($(c)+(225:\spacer)$) rectangle ($(b)+(45:\spacer)$);
\clip[rounded corners=10pt] ($(c)+(225:\spacer)$) rectangle ($(b)+(45:\spacer)$);
\draw[densely dashed] ($(n1)+(135:\estretch)$)--++(315:2*\estretch);
\draw ($(n1)+(45:\estretch)$)--++(225:2*\estretch);
\draw (b)--(d)--(c);
\draw[densely dashed] (c)--(a)--(b);
\draw (b)--++(90:\spacer) (c)--++(180:\spacer);
\draw[densely dashed] (b)--++(0:\spacer) (c)--++(270:\spacer);
\draw[dotted] (a)--++(115:\spacer) (a)--++(135:\spacer) (a)--++(155:\spacer);
\draw[dotted] (d)--++(295:\spacer) (d)--++(315:\spacer) (d)--++(335:\spacer);
\foreach \name\type in {a/\maxx, b/\sadd, c/\sadd, d/\minn, n1/\sadd, n2/\maxx}{
  \type{\name}
}
\end{scope}
\begin{scope}[shift={(\vfactor,0)}]
\node (lab) at (\felen/2,-.8) {(b) Isolated saddle-min};
\node[lab2] at (lab.south) {edge};
\foreach \x\y\name in {0/\felen/a, \felen/\felen/b, 0/0/c, \felen/0/d, .35*\felen/.65*\felen/n1, .65*\felen/.35*\felen/n2}{
  \coordinate (\name) at (\x,\y);
}
\fill[rounded corners=10pt,cellfill] ($(c)+(225:\spacer)$) rectangle ($(b)+(45:\spacer)$);
\clip[rounded corners=10pt] ($(c)+(225:\spacer)$) rectangle ($(b)+(45:\spacer)$);
\draw[densely dashed] (a)--(n1) .. controls +(-45:.4) and +(270:.4) .. ++(0:.75) to [out=90,in=-22.5] (a);
\draw ($(n1)+(45:\estretch)$)--++(225:2*\estretch);
\draw (b)--(d)--(c);
\draw[densely dashed] (c)--(a)--(b);
\draw (b)--++(90:\spacer) (c)--++(180:\spacer);
\draw[densely dashed] (b)--++(0:\spacer) (c)--++(270:\spacer);
\draw[dotted] (a)--++(115:\spacer) (a)--++(135:\spacer) (a)--++(155:\spacer);
\draw[dotted] (d)--++(295:\spacer) (d)--++(315:\spacer) (d)--++(335:\spacer);
\foreach \name\type in {a/\maxx, b/\sadd, c/\sadd, d/\minn, n1/\sadd, n2/\maxx}{
  \type{\name}
}
\end{scope}
\begin{scope}[shift={(2*\vfactor,0)}]
\node (lab) at (\felen/2,-.8) {(c) Isolated saddle-min};
\node[lab2] at (lab.south) {edge};
\foreach \x\y\name in {0/\felen/a, \felen/\felen/b, 0/0/c, \felen/0/d, .35*\felen/.65*\felen/n1, .65*\felen/.35*\felen/n2}{
  \coordinate (\name) at (\x,\y);
}
\fill[rounded corners=10pt,cellfill] ($(c)+(225:\spacer)$) rectangle ($(b)+(45:\spacer)$);
\clip[rounded corners=10pt] ($(c)+(225:\spacer)$) rectangle ($(b)+(45:\spacer)$);
\draw[densely dashed] (n2)--(n1) .. controls +(135:.3) and +(180:.3) .. ++(90:.55) to [out=0,in=90] (n2);
\draw ($(n1)+(45:\estretch)$)--++(225:2*\estretch);
\draw (b)--(d)--(c);
\draw[densely dashed] (c)--(a)--(b);
\draw (b)--++(90:\spacer) (c)--++(180:\spacer);
\draw[densely dashed] (b)--++(0:\spacer) (c)--++(270:\spacer);
\draw[dotted] (a)--++(115:\spacer) (a)--++(135:\spacer) (a)--++(155:\spacer);
\draw[dotted] (d)--++(295:\spacer) (d)--++(315:\spacer) (d)--++(335:\spacer);
\foreach \name\type in {a/\maxx, b/\sadd, c/\sadd, d/\minn, n1/\sadd, n2/\maxx}{
  \type{\name}
}
\end{scope}
\begin{scope}[shift={(\vfactor,-\vfactor)}]
\node (lab) at (\felen/2,-.8) {(d) Unique resolution};
%
\foreach \x\y\name in {0/\felen/a, \felen/\felen/b, 0/0/c, \felen/0/d, .35*\felen/.65*\felen/n1, .65*\felen/.35*\felen/n2}{
  \coordinate (\name) at (\x,\y);
}
\fill[rounded corners=10pt,cellfill] ($(c)+(225:\spacer)$) rectangle ($(b)+(45:\spacer)$);
\clip[rounded corners=10pt] ($(c)+(225:\spacer)$) rectangle ($(b)+(45:\spacer)$);
\draw[densely dashed] (a)--(n2);
\draw (n1) to [bend left=45] (d);
\draw (n1) to [bend right=45] (d);
\draw (b)--(d)--(c);
\draw[densely dashed] (c)--(a)--(b);
\draw (b)--++(90:\spacer) (c)--++(180:\spacer);
\draw[densely dashed] (b)--++(0:\spacer) (c)--++(270:\spacer);
\draw[dotted] (a)--++(115:\spacer) (a)--++(135:\spacer) (a)--++(155:\spacer);
\draw[dotted] (d)--++(295:\spacer) (d)--++(315:\spacer) (d)--++(335:\spacer);
\foreach \name\type in {a/\maxx, b/\sadd, c/\sadd, d/\minn, n1/\sadd, n2/\maxx}{
  \type{\name}
}
\end{scope}
\end{tikzpicture}
\caption{New edges are uniquely determined when a pair of critical points is added in the interior of a quadrangle.}
\label{fig:face-move-proof}
\end{figure}

Next, suppose the pair is added on an edge between two quadrangles. Note that a maximum-saddle pair cannot be added on a separatrix between a minimum and a saddle, and vice-versa, since such a pair would force the saddle in the initial separatrix to have more than four flow lines incident to it. However, adding a maximum-saddle pair to the separatrix between a maximum and a saddle is permissible. In this case, there is a unique way of adding flow lines to obtain a valid configuration, and this is precisely the edge-max move. The edge-min move case is proven similarly.

Finally, suppose the pair is added at a vertex in the quadrangulation. As in the edge moves, a maximum-saddle pair cannot be added at a minimum, and a minimum-saddle pair cannot be added at a maximum. For example, if a minimum-saddle pair was added at a maximum, a quadrangle would be formed with two minima and two saddles, contradicting the Quadrangle Lemma. Suppose a minimum-saddle pair was added at a minimum. Then, the additional saddle must connect to two maxima. If it connects to the same maximum twice, this is a face move. If it connects to two adjacent maxima, meaning they are both connected to a saddle in the quadrangulation, then this is an edge move. Otherwise, the saddle connects to two maxima, and we have the vertex move. The max-saddle case gives rise to the vertex-max move and is proven similarly. 
\qed
\end{proof}

\edits{As the fundamental moves add critical points, the moves change the equivalence class (of the classes described in Section \ref{sec:morsefunc-equivalence}) of the associated Morse--Smale function. However, different moves may give the same change of equivalence class, as witnessed by their Reeb graphs in Figure \ref{fig:fundamental-reeb}. These changes in the Reeb graph are termed \emph{elementary deformations of B-type} and \emph{D-type} by \cite{difabiolandi2016}, though their analysis emphasizes movement of critical values past each other (see Section \ref{sec:tracking-values} for further discussion).}

\begin{figure}[h]\centering
\newcommand\hspacer{3.5}
\begin{tikzpicture}[yscale=1]
\draw[line width=1pt] (-1,0)--(-1,1.5);
\draw[line width=1pt] (0,0)--(0,1.5) (0,.5) .. controls +(0:.4) and +(270:.4) .. (.5,1) --++ (90:.2);
\node at (-.5,.75) {$\leftrightarrow$};
\node at (-.5,-1) {(a) Face-max,};
\node at (-.5,-1.4) {edge-max};
\begin{scope}[shift={(\hspacer,0)}]
\draw[line width=1pt] (-1,0)--(-1,1.5);
\draw[line width=1pt] (0,0)--(0,1.5) (0,1) .. controls +(0:.4) and +(90:.4) .. (.5,.5) --++ (270:.2);
\node at (-.5,.75) {$\leftrightarrow$};
\node at (-.5,-1) {(b) Face-min,};
\node at (-.5,-1.4) {edge-min};
\end{scope}
\begin{scope}[shift={(2*\hspacer,0)}]
\draw[line width=1pt] (-1,0)--(-1,1.5);
\draw[line width=1pt] (0,0)--(0,1.5) (0,.5) .. controls +(0:.4) and +(270:.4) .. (.5,1) --++ (90:.2);
\node at (-.5,.75) {$\leftrightarrow$};
\node at (-.5,-1) {(c) Vertex-max};
\end{scope}
\begin{scope}[shift={(3*\hspacer,0)}]
\draw[line width=1pt] (-1,0)--(-1,1.5);
\draw[line width=1pt] (0,0)--(0,1.5) (0,1) .. controls +(0:.4) and +(90:.4) .. (.5,.5) --++ (270:.2);
\node at (-.5,.75) {$\leftrightarrow$};
\node at (-.5,-1) {(d) Vertex-min};
\end{scope}
\foreach \x in {-1,0,\hspacer-1,\hspacer,3*\hspacer-1,3*\hspacer}{
  \node[anchor=south] at (\x,1.5) {$\vdots$};
}
\foreach \x in {-1,0,\hspacer-1,\hspacer,2*\hspacer-1,2*\hspacer}{
  \node[anchor=north,yshift=5pt] at (\x,0) {$\vdots$};
}
\end{tikzpicture}
\caption{Changes to the Reeb graph by the fundamental moves. }
\label{fig:fundamental-reeb}
\end{figure}
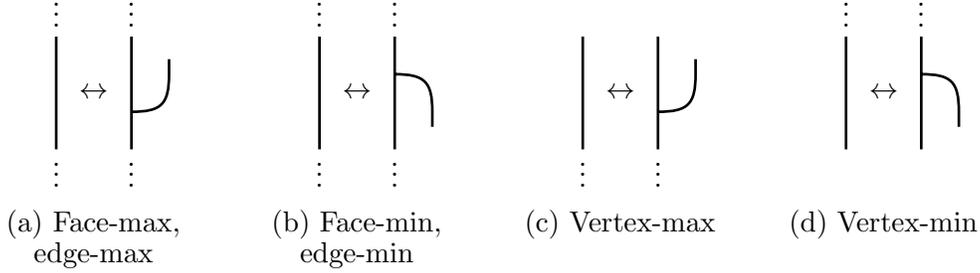

\edits{Moreover, attempting to consider notions of ``distance" on the fundamental moves, as in \cite{damicofrosinilandi2010,bauerlandimemoli2018}, gives no meaningful results, as not specifying critical values allows these distances to be infinitesimally small.}

\begin{remark}
The fundamental moves \edits{also have} implications for the space of all vector fields, with topology induced by distance between functions. This infinite-dimensional stratified space has strata within which vector fields are Morse--Smale, and the \edits{lower-dimensional strata} are where transversality of the stable and unstable manifolds fails to hold. Hence the described operations identify different types of \edits{lower-dimensional} strata, and may be used to count the number of strata within some parameters.
\end{remark}

\section{The Nesting Poset}
\label{sec:results-posets}

To study poset equivalence, we employ the nesting poset introduced in Section~\ref{sec:background}. The nesting poset is used to study the level sets of a Morse function, and describe relations among the posets by using the topology of associated level sets. We also show that the nesting poset is a circle containment order, following research in geometric containment orders~\cite{FishburnTrotter1999}.

\subsection{The Nesting Poset of Level Sets}

We first show that the poset isomorphism condition (2) in Definition~\ref{def:poset-equivalence} can be formulated as a nesting poset isomorphism, or equivalently as a circle containment order isomorphism.

\subsubsection{Nesting Poset of Jordan Curves}
A \emph{Jordan curve} is a non-self-intersecting continuous loop in the plane. Formally, a Jordan curve is a simple closed curve in $\Rspace^2$ that is the image of an injective continuous map $\phi \colon \Sspace^1 \to \Rspace^2$. Let $\gamma \colonequals \image(\phi)$ denote a Jordan curve. The Jordan curve theorem~\cite{Veblen1905} states that the complement $\Rspace^2 - \gamma$ of every Jordan curve $\gamma$ consists of exactly two connected components: one bounded interior component, denoted as $\interior(\gamma)$, and one unbounded exterior component. In this paper, we consider two Jordan curves $\gamma_1$ and $\gamma_2$ to be \emph{semi-disjoint} if they intersect at a single point ($\gamma_1 \cap \gamma_2 = *$). Two nonidentical Jordan curves are \emph{nested} if $\interior(\gamma_1) \subseteq \interior(\gamma_2)$ or vice versa.

Given a set of $m$ Jordan curves $\Gamma \colonequals \{\gamma_1, \ldots, \gamma_m\}$ with at most one pair of semi-disjoint curves (and the rest are disjoint), its complement $\Rspace^2 - \Gamma$ consists of exactly $m+1$ connected components: $m$ bounded components and one unbounded  component.

\begin{remark}
Let $P \colonequals \pi_0(\Rspace^2 - \Gamma)$ denote the set of (path-)connected components of $\Rspace^2 - \Gamma$. The closure of each bounded component in $P$ is a collection of elements of $\Gamma$, consisting of an exterior boundary and zero or more interior boundaries. With a slight abuse of notation, we speak of the boundary of a component in $P$ as the boundary of its closure. Let $p_i \in P$ denote the component whose exterior boundary is~$\gamma_i \in \Gamma$; let $p_0 \in P$ denote the unbounded component. Let $\bdr(p_i)$ denote the set of boundary curves of $p_i$, where we note that $\bdr(p_0)$ contains only interior boundaries. Two components $p_i, p_j \in P$ are \emph{adjacent} if they share a  boundary in $\Gamma$, that is, $\bdr(p_i) \cap \bdr(p_j) \in \Gamma$.
\end{remark}

\begin{definition}[Nesting Poset of Jordan Curves]
Let $\Gamma=\{\gamma_1, \ldots, \gamma_n\}$ and $P = \pi_0(\Rspace^2 - \Gamma)$ be as above. For any two adjacent components $p_i, p_j \in P$, define a binary relation~$\leq$, such that $p_i \leq p_j$ if and only if (1) $\interior(\gamma_i) \subseteq \interior(\gamma_j)$, or (2) $p_j$ is unbounded. The \emph{nesting poset} associated to $\Gamma$ is $N(\Gamma) \colonequals (P, \leq_P)$, where $\leq_P$ is the transitive closure of $\leq$ on $P$.
\end{definition}

Reflexivity, anti-symmetry, and transitivity of $\leq_P$ follow from the same properties of set containment $\subseteq$, so $N(\Gamma)$ is indeed a poset. Figure~\ref{fig:nesting} illustrates five examples of Jordan curves (in white) on the plane. Curves in Figure~\ref{fig:nesting}(a), Figure~\ref{fig:nesting}(b), and  Figure~\ref{fig:nesting}(d) are disjoint; while curves in Figure~\ref{fig:nesting}(c) and Figure~\ref{fig:nesting}(e) are semi-disjoint. For each set $\Gamma$, the nesting poset $N(\Gamma)$ is visualized by its Hasse diagram: each vertex corresponds to an element in $P$ (a green shaded region); each arrow indicates a binary relation between adjacent elements (that is, an arrow exists from $p_i$ to $p_j$ if and only if $p_i \leq p_j$).

\begin{figure}[!ht]\centering
\begin{tikzpicture}[scale=1,
  circ/.style={line width=2pt,draw=white},
  conx/.style={line width=1.5pt,decoration={markings,mark=at position 0.5 with \arrow{Straight Barb[width=8pt,length=4pt]}},postaction=decorate}]
\fill[newcol] (0,0) rectangle (4,3);
\foreach \x\y\lab in {1.1/1.5/a, 2.9/1.5/b, 2/2.5/c}{
  \coordinate (\lab) at (\x,\y);
  \fill (\lab) circle (.1);
}
\draw[circ] (a) circle (.75);
\draw[circ] (b) circle (.75);
\draw[conx] (a)--(c);
\draw[conx] (b)--(c);
\node at (2,-.5) {(a)};
\begin{scope}[shift={(5,0)}]
\fill[newcol] (0,0) rectangle (4,3);
\foreach \x\y\lab in {1.8/1.5/a, 2.7/1.5/b, 3.6/1.5/c}{
  \coordinate (\lab) at (\x,\y);
  \fill (\lab) circle (.1);
}
\draw[circ] (2,1.5) circle (1.2);
\draw[circ] (1.8,1.5) circle (.5);
\draw[conx] (a)--(b);
\draw[conx] (b)--(c);
\node at (2,-.5) {(b)};
\end{scope}
\begin{scope}[shift={(-2.5,-4)}]
\fill[newcol] (0,0) rectangle (4,3);
\foreach \x\y\lab in {1.25/1.5/a, 2.75/1.5/b, 2/2.5/c}{
  \coordinate (\lab) at (\x,\y);
  \fill (\lab) circle (.1);
}
\draw[circ] (a) circle (.75);
\draw[circ] (b) circle (.75);
\draw[conx] (a)--(c);
\draw[conx] (b)--(c);
\node at (2,-.5) {(c)};
\end{scope}
\begin{scope}[shift={(2.5,-4)}]
\fill[newcol] (0,0) rectangle (4,3);
\foreach \x\y\lab in {2.6/1.5/b, 3.6/1.5/c}{
  \coordinate (\lab) at (\x,\y);
  \fill (\lab) circle (.1);
}
\fill[white] (2,1.5) circle (3pt);
\draw[circ] (2,1.5) circle (1.2);
\draw[conx] (b)--(c);
\node at (2,-.5) {(d)};
\end{scope}
\begin{scope}[shift={(7.5,-4)}]
\fill[newcol] (0,0) rectangle (4,3);
\foreach \x\y\lab in {1.3/1.5/a, 2.2/1.5/b, 3.6/1.5/c}{
  \coordinate (\lab) at (\x,\y);
  \fill (\lab) circle (.1);
}
\draw[circ] (2,1.5) circle (1.2);
\draw[circ] (1.3,1.5) circle (.5);
\draw[conx] (a)--(b);
\draw[conx] (b)--(c);
\node at (2,-.5) {(e)};
\end{scope}
\end{tikzpicture}
\caption{Poset structures of level sets of some Morse functions for two regular values (a, b), and three critical values, (c, d, and e).}
\label{fig:nesting}
\end{figure}
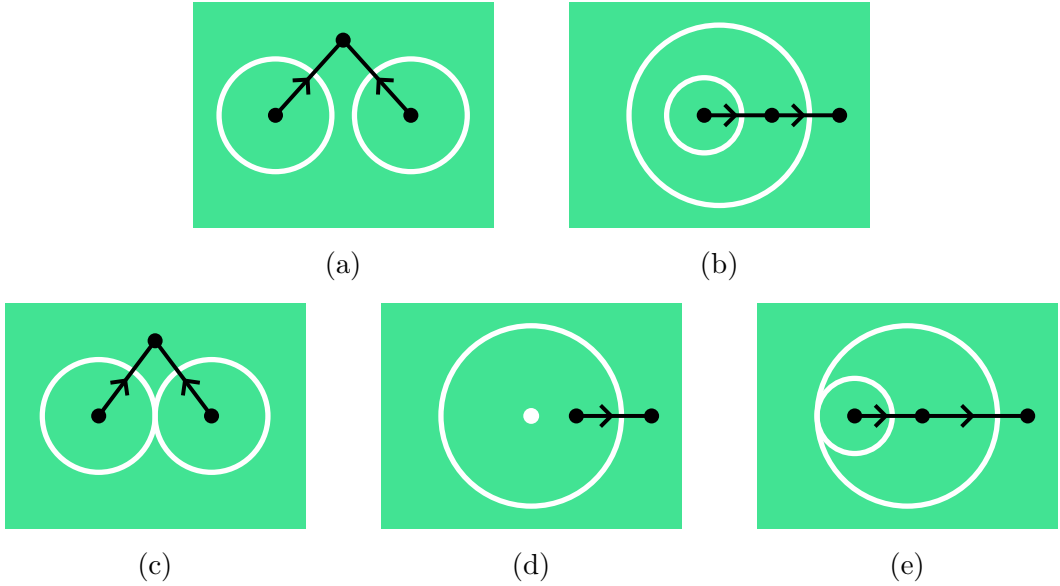

\begin{figure}[!ht]\centering
\begin{tikzpicture}
\foreach \y in {0,...,3}{
  \draw[opacity=.5,dotted] (-5.5,\y)--(5.5,\y);
}
\begin{scope}[shift={(-.9,0)}]
\draw[barco] (-.15,1) -- (-.15,2);
\draw[barcc] (.15,0) -- (.15,3);
\end{scope}
\begin{scope}[shift={(.6,0)}]
\draw[|->] (0,-.75)--(0,4);
\node (r) at (0,-1.25) {$\Rspace$};
\foreach \y\lab in {-.5/{a_0}, .5/{a_1}, 1.5/{a_2}, 2.5/{a_3}, 3.5/{a_4}}{
  \draw (-.2,\y)--(.2,\y);
  \node[anchor=west] at (.2,\y) {$\lab$};
}
\foreach \y\lab in {0/{t_1}, 1/{t_2}, 2/{t_3}, 3/{t_4}}{
  \draw (-.2,\y)--(.2,\y);
  \node[anchor=east,fill=white] at (-.2,\y) {$\lab$};
}
\end{scope}
\begin{scope}[shift={(\bwid-2.5,0)}]
\node at (-1*\bwid,-.75) {$\textup{im}(\iota)$};
\draw[fill=newcol] (0,3) arc (90:0:\bwid/2) --++ (270:3-\bwid)
  arc (0:-180:\bwid/2) --++ (90:2-\bwid)
  arc (0:180:\bwid/2) --++ (270:1-\bwid)
  arc (0:-180:\bwid/2) --++ (90:1-\bwid/2)
  arc (180:90:\bwid) arc (-90:0:\bwid) --++ (90:1-2.5*\bwid)
  arc (180:90:\bwid/2);
\foreach \x\y in {0/1.5, 2*\bwid/.5, 2*\bwid/1.5, 2*\bwid/2.6}{
  \draw (\x-2*\bwid,\y) ellipse (\bwid/2 and \bwid/4);
}
\end{scope}
\begin{scope}[shift={(1*\bwid+2.5,0)}]
\node at (-1*\bwid,-.75) {$\textup{im}(\iota')$};
\draw[fill=newcol] (0,3) arc (90:0:\bwid/2) --++ (270:3-2*\bwid)
  arc (0:-180:1.5*\bwid) --++ (90:2-2*\bwid)
  arc (180:90:\bwid/2) arc (-90:0:1.5*\bwid) --++ (90:1-2*\bwid)
  arc (180:90:\bwid/2);
\draw[fill=newcol!66,dashed] (-2*\bwid,2) arc (90:0:\bwid/2) --++ (270:1-\bwid)
  arc (-180:0:\bwid/2) --++ (90:1+\bwid)
  arc (0:-90:1.5*\bwid);
\foreach \x\y in {\bwid/1.5, 2*\bwid/2.6}{
  \draw (\x-2*\bwid,\y) ellipse (\bwid/2 and \bwid/4);
}
\draw (-\bwid,1.5) ellipse (\bwid*1.5 and .6*\bwid);
\draw (-\bwid,.5) ellipse (\bwid*1.5 and .6*\bwid);
\end{scope}
\begin{scope}[shift={(4.5,0)}]
\foreach \x\y in {0/-.5, 0/0, -.3/.5, .3/.5, -.3/1, .3/1, -.6/1.5, 0/1.5, .6/1.5, -.6/2, 0/2, .6/2, -.3/2.5, .3/2.5, 0/3, 0/3.5}{ 
  \fill (\x,\y) circle (.07);
} 
\foreach \lev\start\end in {.5/-.3/.3, 1/-.3/.3, 1.5/-.6/0, 1.5/0/.6, 2/-.6/0, 2/0/.6, 2.5/-.3/.3}{
  \draw[conx] (\start,\lev)--(\end,\lev);
}
\end{scope}
\begin{scope}[shift={(-4.5,0)}]
\foreach \x\y in {0/-.5, 0/0, -.3/.5, .3/.5, -.3/1, .3/1, -.6/1.5, 0/1.5, .6/1.5, -.6/2, 0/2, .6/2, -.3/2.5, .3/2.5, 0/3, 0/3.5}{ 
  \fill (\x,\y) circle (.07);
} 
\foreach \lev\start\end in {.5/-.3/.3, 1/-.3/.3, 1.5/-.6/0, 1.5/.6/0, 2/-.6/0, 2/.6/0, 2.5/-.3/.3}{
  \draw[conx] (\start,\lev)--(\end,\lev);
}
\end{scope}
\end{tikzpicture}
\caption{A Morse function factoring through two different embeddings $\iota$, $\iota'$, which are height equivalent but not poset equivalent, as distinguished by their different nesting posets.}
\label{fig:shotglassworm}
\end{figure}
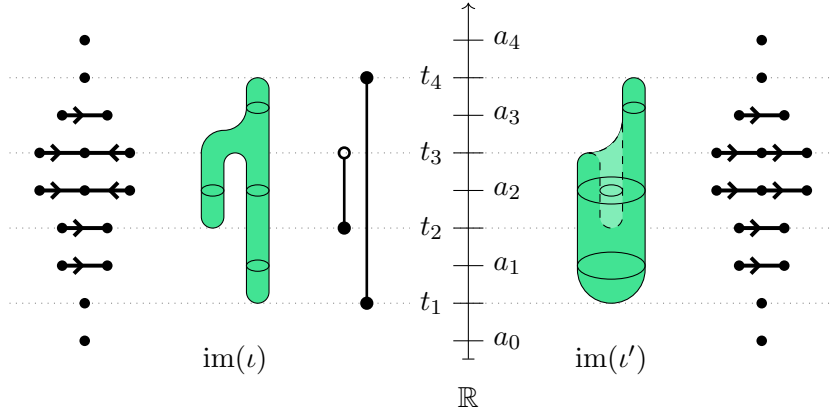

\subsubsection{Nesting Poset of Level Sets}
Let $f\colon \Sspace^2 \rightarrow \Rspace$ be a Morse function that factors
through a smooth embedding $\iota$; that is, $f$ is defined as the composition
$f = \pi \circ \iota$, for $\Sspace^2 \xrightarrow{\iota} \Rspace^3
\xrightarrow{\pi} \Rspace$. Assuming~$\Mspace$ is smooth and compact, the level
set~$f^{-1}(a)$ is a (not necessarily connected) 1-manifold without boundary for a regular value~$a \in \Rspace$ according to the Implicit Function Theorem. The set~$\iota \circ f^{-1}(a)$ is therefore a set of disjoint Jordan curves in the plane $\pi^{-1}(a) \subset \Rspace^3$. For $a = c$ a critical value of $f$, the set~$\iota \circ f^{-1}(a)$ contains either exactly one pair of semi-disjoint Jordan curves or a point (together with other disjoint Jordan curves).

\begin{definition}[Nesting Poset of Level Sets]
\label{def:levelset-nesting}
For any $a \in \Rspace$, let $\Gamma = \iota \circ f^{-1}(a)$ and~$P = \pi_0(\pi^{-1}(a) - \Gamma)$. The \emph{nesting poset} $N_a$ associated with $a$ is the nesting poset of Jordan curves~$\Gamma$ in the plane $\pi^{-1}(a)$. With an abuse of notation, $N_a \colonequals N(\Gamma) = (P, \leq_P)$.
\end{definition}

For example, Figure~\ref{fig:nesting}(a) and Figure~\ref{fig:nesting}(b) illustrate two sets of disjoint Jordan curves (in white) that arise from level sets of two Morse functions $f$ and $g$ at a shared regular value, respectively. Figure~\ref{fig:shotglassworm} describes these examples in the context of their corresponding Morse functions. Specifically, two Morse functions $f, g \colon \Sspace^2 \to \Rspace$ factor through embeddings $\iota, \iota' \colon \Sspace^2 \to \Rspace^3$ with the same barcode, $f = \pi \circ \iota$ and $g = \pi \circ \iota'$ in Figure~\ref{fig:shotglassworm}. For \emph{any} common slicing $a_0 < a_1 < a_2 < a_3 < a_4$, let $F_i \colonequals \pi_0\left(\pi^{-1}(a_i) - \iota \circ f^{-1}(a_i)\right)$ and $G_i \colonequals \pi_0\left(\pi^{-1}(a_i) - \iota' \circ g^{-1}(a_i)\right)$; the map $F_i \to G_i$ is not a poset isomorphism for $i = 2$; in particular, regular value $a_2$ gives rise to a poset structure in Figure~\ref{fig:nesting}(a) for $f$ and a different one in Figure~\ref{fig:nesting}(b) for $g$. By Definition~\ref{def:poset-equivalence}, $f$ and $g$ are not poset equivalent.

\subsubsection{Circle Containment Order}
A partially ordered set $(P, \leq)$ is called a \emph{circle containment order}~\cite{ScheinermanWierman1988}, provided one can assign to each $p_i \in P$ a closed disc in the plane $o_i \subseteq  \Rspace^2$ satisfying $p_i \leq p_j$ if and only if $o_i \subseteq o_j $. Let $\alpha \colon P \to \Rspace^2$ denote such an assignment. If $\phi\colon \Rspace^2 \to \Rspace^2$ is an orientation-preserving homeomorphism, then $\phi$ does not change the circle containment order, so that $\alpha$ and $\phi \circ \alpha$
are equivalent circle containment orders.

The nesting poset structure of $P = \pi_0\left(\pi^{-1}(a) - \iota \circ f^{-1}(a)\right)$, for any regular value $a$, could be understood in terms of a circle containment order. As illustrated in Figure~\ref{fig:nesting}(a), each bounded element in a poset (a green shaded region) $p_{i} \in P$ (where $i>0$) can be assigned a closed disc in the plane, which is the bounded interior component $\interior(\gamma_i)$. The unbounded component $p_0 \in P$ is assigned a closed disc that enclose all other discs. Such an assignment imposes a circle containment order. Hence a nesting poset is a circle containment order.

\subsection{A Morse-Theoretic Perspective on Nesting Posets}

Now we provide theorems analogous to Theorems~\ref{theorem:CMT-A}
and~\ref{theorem:CMT-B}. As before, the Morse function $f\colon \Sspace^2\to \Rspace$ factors through an embedding $\iota$ as $f = \pi\circ \iota$.  Let $L_t \colonequals  \iota \circ f^{-1}(t)$ denote the embedding of its level sets. There are three types of critical points in $\image(\iota)$: local minima, saddles, and local maxima, with indices $0$, $1$, and $2$,  respectively. To study local structure surrounding the critical points, we further classify the saddles into \emph{merging saddles} and \emph{splitting saddles} by investigating the relation between level sets $L_{c-\epsilon}$ and $L_{c+\epsilon}$ as $t$ crosses the critical value $c = f(p)$ of a saddle $p$. If $\epsilon>0$ is small enough, the intervals  $[c-\epsilon, c)$ and $(c, c+\epsilon]$ contain no critical values.

A saddle $p$ is a \emph{merging saddle} if a pair of disjoint Jordan curves in~$L_{c-\epsilon}$ merges into a single Jordan curve in~$L_{c + \epsilon}$ as $t$ crosses $c$. A saddle $p$ is a \emph{splitting saddle} if a Jordan curve at $L_{c-\epsilon}$ splits into a pair of disjoint Jordan curves at $L_{c+\epsilon}$ as $t$ crosses $c$. A merging saddle is of \emph{nesting} type if the Jordan curves that merge at $L_{c + \epsilon}$ are nested at~$L_{c-\epsilon}$; otherwise, the saddle is of \emph{non-nesting} type. Similarly, we can define splitting saddles of nesting and non-nesting types. Figures~\ref{fig:crit-types}~-~\ref{fig:nesting4} illustrate the reasoning behind this terminology.

\begin{figure}[h]\centering
\begin{tikzpicture}[scale=.8]
\foreach \y in {-1.75,0,1.75}{
  \draw[opacity=.5,dotted] (-2.5,\y)--(10,\y);
}
\node at (0,-2.94) {(a)};
  \begin{scope}[shift={(-2,-1.75)}]
  \draw[barcinf] (0,1.75)--(0,3.5);
  \end{scope} 
\fill[newcol] (0,0) circle (.06);
\draw[levels] (0,1.75) ellipse (1.5 and .5);
\draw (-1.5,1.75) to [out=270,in=180] (0,0) to [out=0,in=270] (1.5,1.75);
\draw (0,1.75) ellipse (1.5 and .5);

\begin{scope}[shift={(5.5,0)}]
\node at (0,-3) {(b)};
  \begin{scope}[shift={(-2.75,-1.75)}]
  \draw[infb] (.2,0)--(.2,3.5);
  \draw[barcinfo] (-.2,1.75)--(-.2,0);
  \end{scope} 
\draw[levels] (1,-1.75) ellipse (.75 and .25);
\draw (.25,-1.75) arc (180:360:.75 and .25);
\draw[dashed] (.25,-1.75) arc (180:09:.75 and .25);
\draw[levels] (-1,-1.75) ellipse (.75 and .25);
\draw (-1.75,-1.75) arc (180:360:.75 and .25);
\draw[dashed] (-1.75,-1.75) arc (180:0:.75 and .25);
\draw[levels] (1,0) ellipse (1 and .3);
\draw (0,0) arc (180:360:1 and .3);
\draw[dashed] (0,0) arc (180:0:1 and .3);
\draw[levels] (-1,0) ellipse (1 and .3);
\draw (-2,0) arc (180:360:1 and .3);
\draw[dashed] (-2,0) arc (180:0:1 and .3);
\draw[levels] (0,1.75) ellipse (2 and .6);
\draw (0,1.75) ellipse (2 and .6);
\draw (-1.75,-1.75)--(-2,0)--(-2,1.75) (1.75,-1.75)--(2,0)--(2,1.75);
\draw (-.25,-1.75) .. controls +(85:1) and +(190:.2) .. (0,0) .. controls +(-10:.2) and +(95:1) .. (.25,-1.75);
\end{scope}
\begin{scope}[shift={(11.5,0)}]
\node at (0,-3) {(c)};
  \begin{scope}[shift={(-2.75,-1.75)}]
  \draw[infb] (.2,0)--(.2,3.5);
  \draw[barcinfo] (-.2,1.75)--(-.2,0);
  \end{scope}
\draw[levels] (0,-1.75) ellipse (2 and .6);
\draw[levels] (0,-1.75) ellipse (.75 and .25);
\draw (-2,-1.75) arc (180:360:2 and .6);
\draw[dashed] (-2,-1.75) arc (180:0:2 and .6);
\draw[dashed] (0,-1.75) ellipse (.75 and .25);
\draw[levels] (0,0) ellipse (2 and .6);
\draw (-2,0) arc (180:360:2 and .6);
\draw[dashed] (-2,0) arc (180:0:2 and .6);
\draw[levels] (1.1,0) ellipse (.9 and .3);
\draw[dashed] (1.1,0) ellipse (.9 and .3);
\draw[levels] (-.5,1.75) ellipse (1.5 and .5);
\draw (-.5,1.75) ellipse (1.5 and .5);
\draw (-2,-1.75)--(-2,1.75);
\draw (2,-1.75)--(2,0) .. controls +(90:.7) and +(270:1.5) .. (1,1.75);
\draw[dashed] (-.75,-1.75) to [out=90,in=270] (1,1.75);
\draw[dashed] (.75,-1.75) to [out=90,in=250] (2,0);
\end{scope}
\begin{scope}[shift={(-2.6,0)}]
\node[anchor=east] at (0,1.75) {$L_{c+\epsilon}$};
\node[anchor=east] at (0,0) {$L_{c}$};
\node[anchor=east] at (0,-1.75) {$L_{c-\epsilon}$};
\end{scope}
\end{tikzpicture}
\caption{A local minimum (a), a non-nesting (b) and a nesting (c) merging saddle
    with their corresponding (partial) zero-dimensional interlevel persistence barcodes.}
\label{fig:crit-types}
\end{figure}
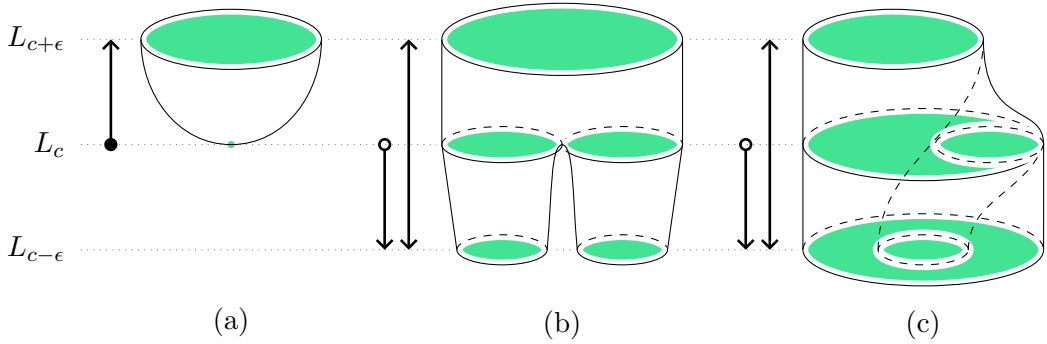

As before, let $\Gamma = L_t = \iota \circ f^{-1}(t)$ and $N_t = N(\Gamma) = (P, \leq_P)$ be the nesting poset. We now study how $N_t$ changes as $t\in \Rspace$ changes.

\begin{theorem}
\label{theorem:nesting-poset-A}
If $f$ has no critical values in the interval $[a,b]$, then $N_a$ and $N_b$ are poset isomorphic, that is, $N_a \cong N_b$.
\end{theorem}

\begin{proof}
This follows from a key observation in proving Theorem~\ref{theorem:CMT-A} from Morse theory.
Recall~\cite[Theorem 2.31]{Matsumoto1997}, that if $f$ has no critical values in the interval $[a, b]$, then $\Mspace_{[a,b]}:=\{x \in \Mspace \mid a \leq f(a) \leq b\}$ is diffeomorphic to the product $f^{-1}(a) \times [a,b]$. Using the gradient-like vector field for $f$, the proof of Theorem 2.31 in~\cite{Matsumoto1997} includes a construction of an orientation-preserving diffeomorphism $h\colon f^{-1}(a) \times [0, b-a] \to \Mspace_{[a,b]}$. Therefore, $\Mspace_{[a,b]}$ is diffeomorphic to $f^{-1}(a) \times [0, b-a]$ and thus also to $f^{-1}(a) \times [a,b]$.

However, we do not study $f$ directly, we instead study the function $\pi$ restricted to $\image(\iota)$ and sublevel sets $L_{[a,b]} \colonequals \iota \circ f^{-1}[a,b]$. Nonetheless, Theorem 2.31 from~\cite{Matsumoto1997} still applies, that is, there exists a diffeomorphism $h \colon L_a \times [0, b-a] \to L_{[a,b]}$ implying $L_{[a,b]} \cong L_{a} \times [0, b-a]$. The diffeomorphism $h$ is orientation-preserving, therefore it does not change the circle containment order moving from $N_a$ to $N_b$. Therefore $N_a \cong N_b$.
\qed
\end{proof}

Denote an injective map of posets by $\hookrightarrow$ and a surjective map of posets by $\twoheadrightarrow$. If the injective map happens to be an isomorphism, we write $\xhookrightarrow{\cong}$. The conditions on $f\colon \Sspace^2\to \Rspace$ are as above.

\begin{theorem}
\label{theorem:nesting-poset-B}
Let $p$ be a critical point of $f$ with critical value $c \colonequals f(p)$. Then for $\epsilon$ small enough, there exist zigzags of poset maps:
\begin{enumerate}
\item\label{thm:nesting1} $N_{c-\epsilon} \xhookleftarrow \cong N_c \hookrightarrow N_{c+\epsilon}$ if $p$ is a local minimum;
\item\label{thm:nesting2} $N_{c-\epsilon} \hookleftarrow N_c \xhookrightarrow \cong N_{c+\epsilon}$ if $p$ is a local maximum;
\item\label{thm:nesting3} $N_{c-\epsilon} \xhookleftarrow \cong N_c \twoheadrightarrow N_{c+\epsilon}$ if $p$ is a non-nesting merging saddle;
\item\label{thm:nesting4} $N_{c-\epsilon} \twoheadleftarrow N_c \xhookrightarrow \cong N_{c+\epsilon}$ if $p$ is a non-nesting splitting saddle;
\item\label{thm:nesting5} $N_{c-\epsilon} \xhookleftarrow \cong N_c \hookleftarrow N_{c+\epsilon}$ if $p$ is a nesting merging saddle;
\item\label{thm:nesting6} $N_{c-\epsilon} \hookrightarrow N_c \xhookrightarrow \cong N_{c+\epsilon}$ if $p$ is a nesting splitting saddle.
\end{enumerate}
\end{theorem}

As the behavior of nesting posets for local maxima and nesting/non-nesting splitting saddles of $f$ is the same as local minima and nesting/non-nesting merging saddles, respectively, of the Morse function $-f$, we only prove Statements \ref{thm:nesting1}, \ref{thm:nesting3}, and \ref{thm:nesting5} in Theorem~\ref{theorem:nesting-poset-B}.

\begin{proof}

By the Morse Lemma, there exists a neighborhood $V$ of $p$ such that on $V$, $f(x) = f(p) \pm x_1^2 \pm x_2^2$. The function $f$ is excellent, so there exists some $\epsilon > 0$ such that $f^{-1}[c-\epsilon, c+ \epsilon]$ contains only one critical point of $f$, namely $p$. Let $U\colonequals f^{-1}[c-\epsilon,c+\epsilon] \cap V$. Then $\nabla f$ provides a diffeomorphism from $f^{-1}(c-\epsilon) \setminus U$ to $f^{-1}(c+\epsilon) \setminus U$. By Theorem~\ref{theorem:nesting-poset-A}, the nesting poset is unchanged outside of $U$, hence the proof is reduced to a local computation of the nesting poset of $U$ for each case.

For every $t\in [c-\epsilon,c+\epsilon]$, let $L_t \colonequals \iota(f^{-1}(t) \cap V)$ and let $L \colonequals \bigcup_{t\in [c-\epsilon,c+\epsilon]}L_t$. Take a contractible neighborhood $W$ of $L$, and without loss of generality assume that $L\subseteq \pi^{-1}[c-\epsilon,c+\epsilon]$ and $W_t \colonequals \pi^{-1}(t)\cap W$ is non-empty and contractible for every $t\in [c-\epsilon,c+\epsilon]$. Note that $\iota\circ f^{-1}[c-\epsilon,c+\epsilon]$ may contain more than one connected component. This proof proceeds functorially, by describing for every $t\leqslant s\in [c-\epsilon,c+\epsilon]$ a map of posets $N_t\to N_s$ induced by the topological inclusions $W_t-L_t\hookrightarrow (W-L)\cap \pi^{-1}[t,s]$ and $W_s-L_s\hookrightarrow (W-L)\cap \pi^{-1}[t,s]$.

For Statement~\ref{thm:nesting1}, the function $f$ has index $0$ and so $f(x)=f(p)+x_1^2+x_2^2$ in $V$. As in Figure~\ref{fig:nesting1}, assign labels to the connected components of the three (subsets of) level sets $W_{c-\epsilon}\subseteq \pi^{-1}(c-\epsilon)$, $W_c-L_c\subseteq \pi^{-1}(c)$, and $W_{c+\epsilon}-L_{c+\epsilon}\subseteq \pi^{-1}(c+\epsilon)$. The topological inclusion $A'\hookrightarrow A$ and the natural injective map $A'\hookrightarrow A''$ that widens the hole of $A'$ induce analogous maps on the nesting posets $N_c\to N_{c-\epsilon}$ and $N_c\to N_{c+\epsilon}$, respectively. The poset elements are given the same labels as the connected components to which they correspond, and their relation is defined in Definition \ref{def:levelset-nesting}.

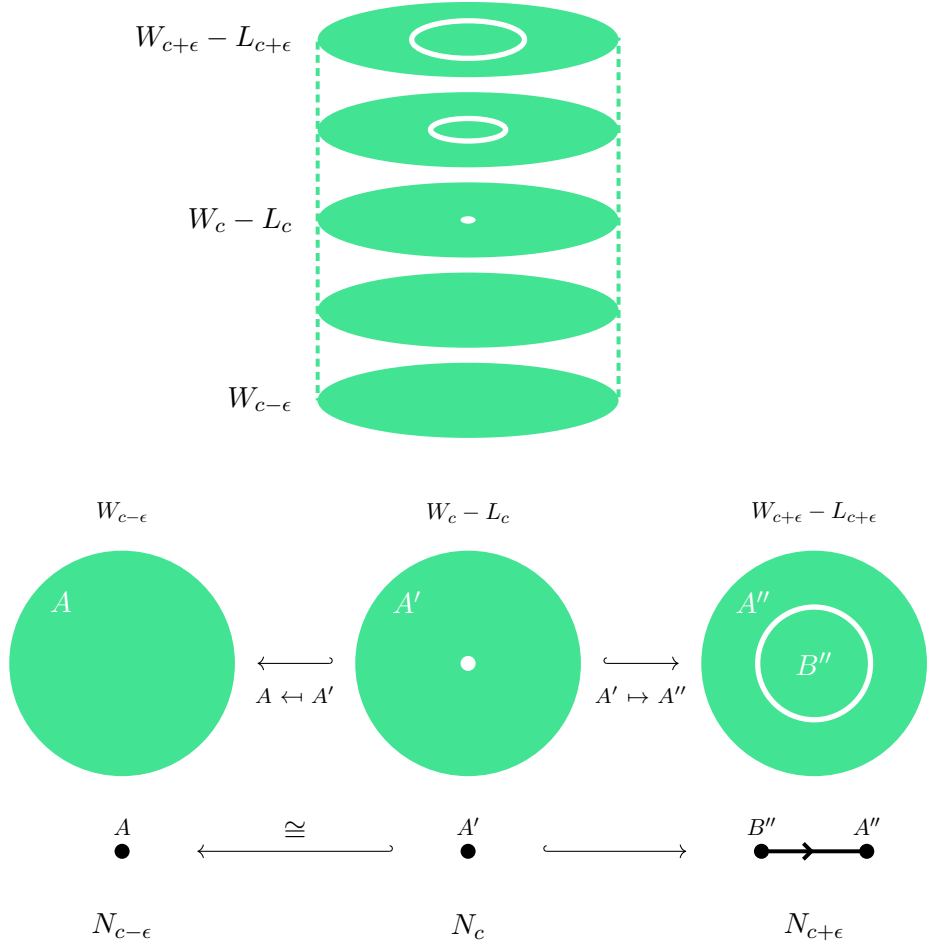
\begin{figure}[!h]\centering
\begin{tikzpicture}[scale=1,
  circ/.style={line width=2pt,draw=white},
  ldisk/.style={fill=newcol},
  tmap/.style={shorten <=1.8cm, shorten >=1.8cm},
  gmap/.style={shorten <=1cm, shorten >=1cm},
  gnode/.style={anchor=south,yshift=3pt,scale=.8}]

\begin{scope}[shift={(4.6,3.5)}]
\draw[newcol,densely dashed,line width=1.5pt] (-2,0) -- (-2,4*\levelspacer) -- (2,4*\levelspacer) -- (2,0) -- cycle;
\node[anchor=east] at (-2.2,0) {$W_{c-\epsilon}$};
\fill[ldisk] (0,0) ellipse (2 and .5);
\fill[ldisk] (0,\levelspacer) ellipse (2 and .5);
\node[anchor=east] at (-2.2,2*\levelspacer) {$W_{c} - L_{c}$};
\fill[ldisk] (0,2*\levelspacer) ellipse (2 and .5);
\fill[white] (0,2*\levelspacer) ellipse (.1 and .05);
\fill[ldisk] (0,3*\levelspacer) ellipse (2 and .5);
\draw[circ] (0,3*\levelspacer) ellipse (.5 and .15);
\node[anchor=east] at (-2.2,4*\levelspacer) {$W_{c+\epsilon} - L_{c+\epsilon}$};
\fill[ldisk] (0,4*\levelspacer) ellipse (2 and .5);
\draw[circ] (0,4*\levelspacer) ellipse (.75 and .25);
\end{scope}
\coordinate (a) at (0,0);
\fill[newcol] (0,0) circle (1.5);
\node[white] at ($(0,0)+(135:1.15)$) {$A$};
\node[scale=.8] at (0,2) {$W_{c-\epsilon}$};
\begin{scope}[shift={(0,-2.5)}]
\coordinate (A) at (0,0);
\fill (A) circle (.1);
\node[gnode] at (A) {$A$};
\node at (0,-1) {$N_{c-\epsilon}$};
\end{scope}
\begin{scope}[shift={(4.6,0)}]
\coordinate (b) at (0,0);
\fill[newcol] (0,0) circle (1.5);
\node[white] at ($(0,0)+(135:1.15)$) {$A'$};
\fill[white] (0,0) circle (.1);
\node[scale=.8] at (0,2) {$W_{c} - L_{c}$};
\begin{scope}[shift={(0,-2.5)}]
\coordinate (B) at (0,0);
\fill (B) circle (.1);
\node[gnode] at (B) {$A'$};
\node at (0,-1) {$N_{c}$};
\end{scope}
\end{scope}
\begin{scope}[shift={(9.2,0)}]
\coordinate (c) at (0,0);
\fill[newcol] (0,0) circle (1.5);
\node[white] at ($(0,0)+(135:1.15)$) {$A''$};
\draw[circ] (0,0) circle (.75);
\node[white] at (0,0) {$B''$};
\node[scale=.8] at (0,2) {$W_{c+\epsilon} - L_{c+\epsilon}$};
\begin{scope}[shift={(0,-2.5)}]
\coordinate (bpp) at (-.7,0);
\coordinate (app) at (.7,0);
\fill (bpp) circle (.1);
\fill (app) circle (.1);
\node[gnode] at (bpp) {$B''$};
\node[gnode] at (app) {$A''$};
\draw[conx] (bpp)--(app);
\node at (0,-1) {$N_{c+\epsilon}$};
\end{scope}
\end{scope}
\draw[tmap,<-{Hooks[left]}] (a) to node[below=5pt] {$\begin{smallmatrix} A & \mapsfrom & A' \end{smallmatrix}$} (b);
\draw[tmap,{Hooks[right]}->] (b) to node[below=5pt] {$\begin{smallmatrix} A' & \mapsto & A'' \end{smallmatrix}$} (c);
\draw[gmap,<-{Hooks[left]}] (A) to node[above] {$\cong$} (B);
\draw[gmap,{Hooks[right]}->] (B) -- (bpp);
\end{tikzpicture}
\caption{Topological construction and labelling for Statement~\ref{thm:nesting1}.}
\label{fig:nesting1}
\end{figure}

For Statements~\ref{thm:nesting3} and \ref{thm:nesting5}, the function $f$ has index 1 and so $f(x)=f(p)-x_1^2+x_2^2$ up to diffeomorphism. As in Figure \ref{fig:nesting2}, assigning labels coherently to the connected components of the three (subsets of) level sets is ambiguous, as some may connect beyond $W$. It is necessary to clarify this, as we want an injective map from the nesting poset constructed from $W_t-L_t$ to the nesting poset constructed from $\pi^{-1}(t)-\iota(f^{-1}(t))$.

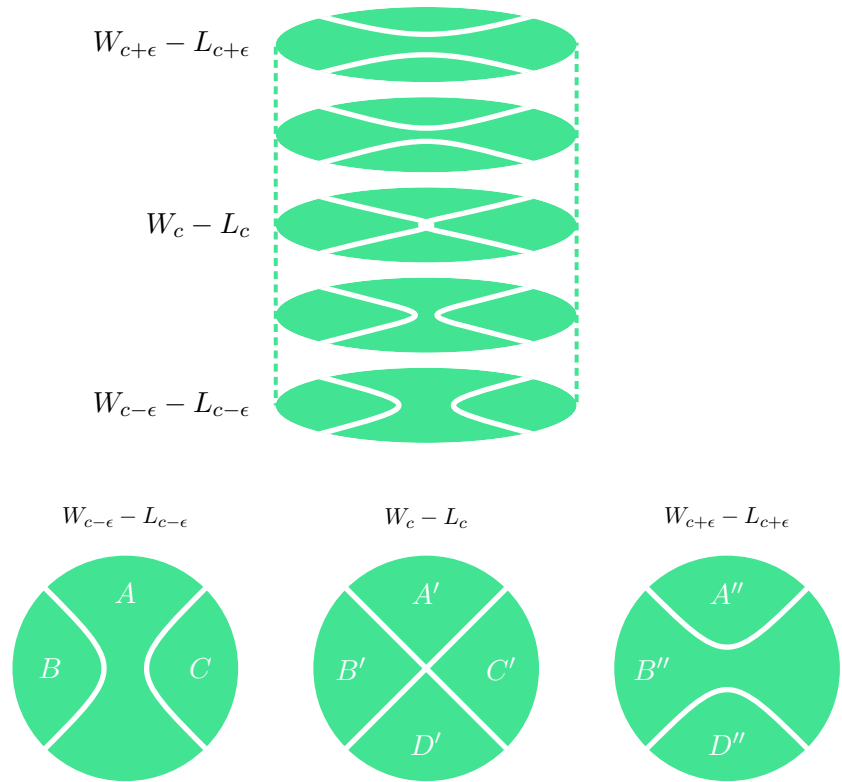
\begin{figure}[!h]\centering
\begin{tikzpicture}[scale=1,
  circ/.style={line width=2pt,draw=white},
  ldisk/.style={fill=newcol}]

\begin{scope}[shift={(4,3.5)}]
\draw[newcol,densely dashed,line width=1.5pt] (-2,0) -- (-2,4*\levelspacer) -- (2,4*\levelspacer) -- (2,0) -- cycle;
\node[anchor=east] at (-2.2,0) {$W_{c-\epsilon}-L_{c-\epsilon}$};
\begin{scope}
\clip (0,0) ellipse (2 and .5);
\fill[ldisk] (0,0) ellipse (2 and .5);
\draw[circ] ($(0,0)+(-15:1.5)$) .. controls +(165:1.5) and +(195:1.5) .. ($(0,0)+(15:1.5)$);
\draw[circ,xscale=-1] ($(0,0)+(-15:1.5)$) .. controls +(165:1.5) and +(195:1.5) .. ($(0,0)+(15:1.5)$);
\end{scope}
\begin{scope}[shift={(0,\levelspacer)}]
\clip (0,0) ellipse (2 and .5);
\fill[ldisk] (0,0) ellipse (2 and .5);
\draw[circ] ($(0,0)+(-15:1.5)$) .. controls +(165:1.8) and +(195:1.8) .. ($(0,0)+(15:1.5)$);
\draw[circ,xscale=-1] ($(0,0)+(-15:1.5)$) .. controls +(165:1.8) and +(195:1.8) .. ($(0,0)+(15:1.5)$);
\end{scope}
\node[anchor=east] at (-2.2,2*\levelspacer) {$W_{c} - L_{c}$};
\begin{scope}[shift={(0,2*\levelspacer)}]
\clip (0,0) ellipse (2 and .5);
\fill[ldisk] (0,0) ellipse (2 and .5);
\draw[circ] ($(0,0)+(-15:1.5)$) .. controls +(165:2.1) and +(195:2.1) .. ($(0,0)+(15:1.5)$);
\draw[circ,xscale=-1] ($(0,0)+(-15:1.5)$) .. controls +(165:2.1) and +(195:2.1) .. ($(0,0)+(15:1.5)$);
\end{scope}
\begin{scope}[shift={(0,3*\levelspacer)}]
\clip (0,0) ellipse (2 and .5);
\fill[ldisk] (0,0) ellipse (2 and .5);
\draw[circ] ($(0,0)+(-15:1.5)$) .. controls +(165:1.55) and +(15:1.55) .. ($(0,0)+(195:1.5)$);
\draw[circ,yscale=-1] ($(0,0)+(-15:1.5)$) .. controls +(165:1.55) and +(15:1.55) .. ($(0,0)+(195:1.5)$);
\end{scope}
\node[anchor=east] at (-2.2,4*\levelspacer) {$W_{c+\epsilon} - L_{c+\epsilon}$};
\begin{scope}[shift={(0,4*\levelspacer)}]
\clip (0,0) ellipse (2 and .5);
\fill[ldisk] (0,0) ellipse (2 and .5);
\draw[circ] ($(0,0)+(-15:1.5)$) .. controls +(165:1.3) and +(15:1.3) .. ($(0,0)+(195:1.5)$);
\draw[circ,yscale=-1] ($(0,0)+(-15:1.5)$) .. controls +(165:1.3) and +(15:1.3) .. ($(0,0)+(195:1.5)$);
\end{scope}
\end{scope}
\coordinate (a) at (0,0);
\fill[newcol] (0,0) circle (1.5);
\node[white] at (0,1) {$A$};
\node[white] at (-1,0) {$B$};
\node[white] at (1,0) {$C$};
\draw[circ] ($(0,0)+(-45:1.6)$) .. controls +(135:1.6) and +(225:1.6) .. ($(0,0)+(45:1.6)$);
\draw[circ,xscale=-1] ($(0,0)+(-45:1.6)$) .. controls +(135:1.6) and +(225:1.6) .. ($(0,0)+(45:1.6)$);
\node[scale=.8] at (0,2) {$W_{c-\epsilon}-L_{c-\epsilon}$};
\begin{scope}[shift={(4,0)}]
\coordinate (b) at (0,0);
\fill[newcol] (0,0) circle (1.5);
\node[white] at (0,1) {$A'$};
\node[white] at (0,-1) {$D'$};
\node[white] at (-1,0) {$B'$};
\node[white] at (1,0) {$C'$};
\draw[circ] (1.5,1.5)--(-1.5,-1.5) (1.5,-1.5)--(-1.5,1.5);
\node[scale=.8] at (0,2) {$W_{c} - L_{c}$};
\end{scope}
\begin{scope}[shift={(8,0)}]
\coordinate (c) at (0,0);
\fill[newcol] (0,0) circle (1.5);
\node[white] at (0,1) {$A''$};
\node[white] at (-1,0) {$B''$};
\node[white] at (0,-1) {$D''$};
\draw[circ] ($(0,0)+(-45:1.6)$) .. controls +(135:1.6) and +(45:1.6) .. ($(0,0)+(225:1.6)$);
\draw[circ,yscale=-1] ($(0,0)+(-45:1.6)$) .. controls +(135:1.6) and +(45:1.6) .. ($(0,0)+(225:1.6)$);
\node[scale=.8] at (0,2) {$W_{c+\epsilon} - L_{c+\epsilon}$};
\end{scope}
\end{tikzpicture}
\caption{Ambiguity in coherent component labeling of (subsets of) level sets near a saddle.}
\label{fig:nesting2}
\end{figure}

To resolve this, take a larger neighborhood $W'\supseteq W$ so that
$W_t'\colonequals W'\cap \pi^{-1}(t)$ contains some connected components of
the embedded 1-manifold $\iota(f^{-1}(t))$, for every $t\in
[c-\epsilon,c+\epsilon]$ a regular value (and contains an embedded
$\Sspace^1\vee \Sspace^1$ for $t$ a critical value). There are 4 unique
embeddings, up to diffeomorphism, as shown in Figure \ref{fig:nesting-ALL},
among which Figure \ref{fig:nesting-ALL}(a) corresponds to
Statement~\ref{thm:nesting3} and Figure \ref{fig:nesting-ALL}(c) corresponds
to Statement~\ref{thm:nesting5}.

For Statement~\ref{thm:nesting3}, as in Figure \ref{fig:nesting3}, assign labels to the connected components of the (subsets of) level sets $W_{c-\epsilon}'-\iota(f^{-1}(c-\epsilon))$, $W_c'-\iota(f^{-1}(c))$, and $W_{c+\epsilon}'-\iota(f^{-1}(c+\epsilon))$. The topological inclusions $A'\hookrightarrow A$, $B'\xhookrightarrow{\cong} B$, and $C'\xhookrightarrow{\cong} C$ induce an analogous nesting poset $N_c\to N_{c-\epsilon}$. Similarly, the topological surjections $A'\xtwoheadrightarrow{\cong}A''$, $B'\xtwoheadrightarrow{\cong}B''$, and $C'\xtwoheadrightarrow{\cong}B''$ induce an analogous nesting poset map $N_c\to N_{c+\epsilon}$.

\begin{figure}[!h]\centering
\begin{tikzpicture}[scale=1,
  circ/.style={line width=2pt,draw=white},
  ldisk/.style={fill=newcol,opacity=.8},
  tmap/.style={shorten <=1.8cm, shorten >=1.8cm},
  gmap/.style={shorten <=1cm, shorten >=1cm},
  gnode/.style={anchor=south,yshift=3pt,scale=.8},
  gnoder/.style={anchor=west,xshift=3pt,scale=.8},
  gnodel/.style={anchor=east,xshift=-3pt,scale=.8}]
\coordinate (a) at (0,0);
\fill[newcol] (0,0) circle (1.5);
\node[white] at (0,1) {$A$};
\node[white] at (-.7,0) {$B$};
\node[white] at (.7,0) {$C$};
\begin{scope}[scale=.5]
\draw[circ] ($(0,0)+(-45:1.6)$) .. controls +(135:1.6) and +(225:1.6) .. ($(0,0)+(45:1.6)$);
\draw[circ,xscale=-1] ($(0,0)+(-45:1.6)$) .. controls +(135:1.6) and +(225:1.6) .. ($(0,0)+(45:1.6)$);
\draw[circ] ($(0,0)+(45:1.6)$) .. controls +(45:2) and +(-45:2) .. ($(0,0)+(-45:1.6)$);
\draw[circ,xscale=-1] ($(0,0)+(45:1.6)$) .. controls +(45:2) and +(-45:2) .. ($(0,0)+(-45:1.6)$);
\end{scope}
\node at (0,1.9) {$W_{c-\epsilon}'-\iota(f^{-1}(c-\epsilon))$};
\begin{scope}[shift={(0,-2.3)}]
\coordinate (A) at (.5,0);
\coordinate (Ab) at (-.5,.4);
\coordinate (Ac) at (-.5,-.4);
\fill (Ab) circle (.08);
\fill (Ac) circle (.08);
\draw[conx] (Ab)--(A);
\draw[conx] (Ac)--(A);
\fill (A) circle (.08);
\node[gnoder] at (A) {$A$};
\node[gnodel] at (Ab) {$B$};
\node[gnodel] at (Ac) {$C$};
\node at (0,-1) {$N_{c-\epsilon}$};
\end{scope}
\begin{scope}[shift={(5,0)}]
\coordinate (b) at (0,0);
\fill[newcol] (0,0) circle (1.5);
\node[white] at (0,1) {$A'$};
\node[white] at (-.7,0) {$B'$};
\node[white] at (.7,0) {$C'$};
\begin{scope}[scale=.5]
\draw[circ] ($(0,0)+(-45:1.6)$) -- ++(135:3.2);
\draw[circ,xscale=-1] ($(0,0)+(-45:1.6)$) -- ++(135:3.2);
\draw[circ] ($(0,0)+(45:1.6)$) .. controls +(45:2) and +(-45:2) .. ($(0,0)+(-45:1.6)$);
\draw[circ,xscale=-1] ($(0,0)+(45:1.6)$) .. controls +(45:2) and +(-45:2) .. ($(0,0)+(-45:1.6)$);
\end{scope}
\node at (0,1.9) {$W_{c}' - \iota(f^{-1}(c))$};
\begin{scope}[shift={(0,-2.3)}]
\coordinate (B) at (.5,0);
\coordinate (Bb) at (-.5,.4);
\coordinate (Bc) at (-.5,-.4);
\fill (Bb) circle (.08);
\fill (Bc) circle (.08);
\draw[conx] (Bb)--(B);
\draw[conx] (Bc)--(B);
\fill (B) circle (.08);
\node[gnoder] at (B) {$A'$};
\node[gnodel] at (Bb) {$B'$};
\node[gnodel] at (Bc) {$C'$};
\node at (0,-1) {$N_{c}$};
\end{scope}
\end{scope}
\begin{scope}[shift={(10,0)}]
\coordinate (c) at (0,0);
\fill[newcol] (0,0) circle (1.5);
\node[white] at (0,1) {$A''$};
\node[white] at (-.7,0) {$B''$};
\begin{scope}[scale=.5]
\draw[circ] ($(0,0)+(-45:1.6)$) .. controls +(135:1.6) and +(45:1.6) .. ($(0,0)+(225:1.6)$);
\draw[circ,yscale=-1] ($(0,0)+(-45:1.6)$) .. controls +(135:1.6) and +(45:1.6) .. ($(0,0)+(225:1.6)$);
\draw[circ] ($(0,0)+(45:1.6)$) .. controls +(45:2) and +(-45:2) .. ($(0,0)+(-45:1.6)$);
\draw[circ,xscale=-1] ($(0,0)+(45:1.6)$) .. controls +(45:2) and +(-45:2) .. ($(0,0)+(-45:1.6)$);
\end{scope}
\node at (0,1.9) {$W_{c+\epsilon}' - \iota(f^{-1}(c+\epsilon))$};
\begin{scope}[shift={(0,-2.3)}]
\coordinate (bpp) at (-.7,0);
\coordinate (app) at (.7,0);
\fill (bpp) circle (.08);
\fill (app) circle (.08);
\node[gnode] at (bpp) {$B''$};
\node[gnode] at (app) {$A''$};
\draw[conx] (bpp)--(app);
\node at (0,-1) {$N_{c+\epsilon}$};
\end{scope}
\end{scope}
\draw[tmap,<-{Hooks[left]}] (a) to node[below=5pt] {$\begin{smallmatrix} A & \mapsfrom & A' \\ B & \mapsfrom & B' \\ C & \mapsfrom & C' \end{smallmatrix}$} (b);
\draw[tmap,->>] (b) to node[below=5pt] {$\begin{smallmatrix} A' & \mapsto & A'' \\ B' & \mapsto & B'' \\ C' & \mapsto & B'' \end{smallmatrix}$} (c);
\draw[gmap,<-{Hooks[left]}] (A) to node[above] {$\cong$} ($(Bb)!.5!(Bc)$);
\draw[gmap,->>] (B) -- (bpp);
\end{tikzpicture}
\caption{Topological construction and labelling for Statement~\ref{thm:nesting3}.}
\label{fig:nesting3}
\end{figure}
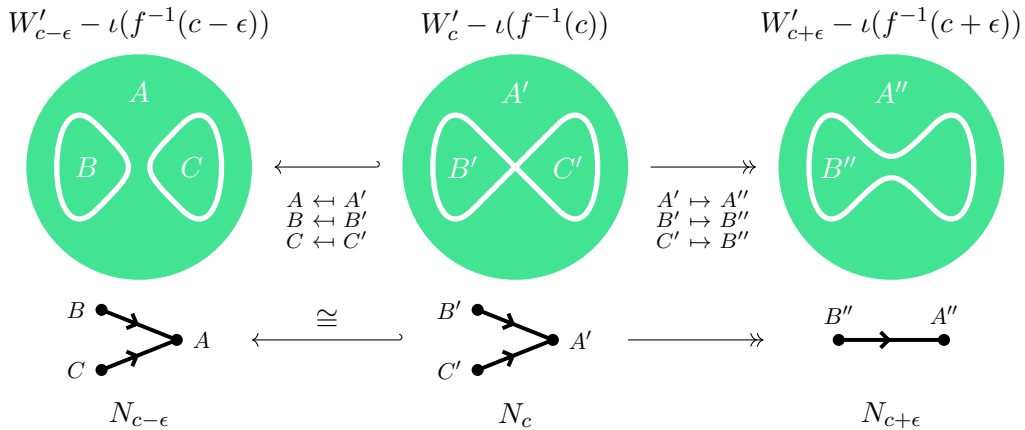

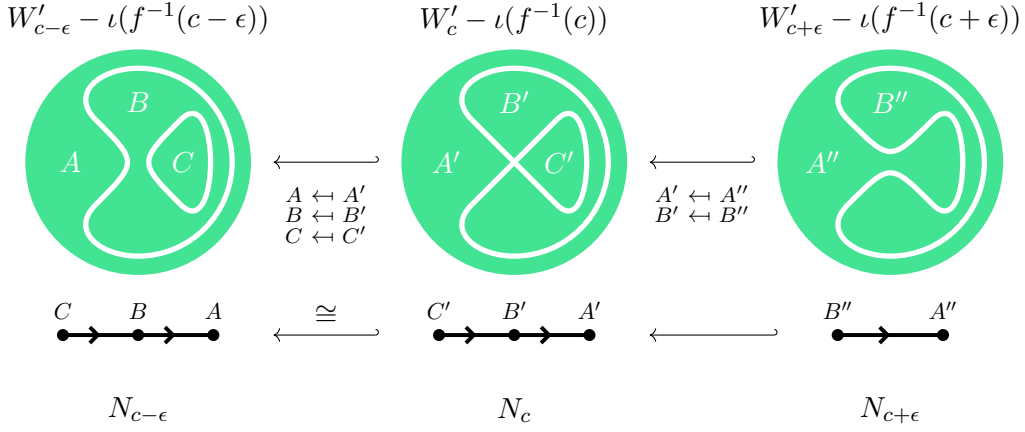
\begin{figure}[!ht]\centering
\begin{tikzpicture}[scale=1,
  circ/.style={line width=2pt,draw=white},
  ldisk/.style={fill=newcol,opacity=.8},
  tmap/.style={shorten <=1.8cm, shorten >=1.8cm},
  gmap/.style={shorten <=.8cm, shorten >=.8cm},
  gnode/.style={anchor=south,yshift=3pt,scale=.8},
  gnoder/.style={anchor=west,xshift=3pt,scale=.8},
  gnodel/.style={anchor=east,xshift=-3pt,scale=.8}]
\coordinate (a) at (0,0);
\fill[newcol] (0,0) circle (1.5);
\node[white] at (0,.8) {$B$};
\node[white] at (-.9,0) {$A$};
\node[white] at (.6,0) {$C$};
\begin{scope}[scale=.5]
\draw[circ] ($(0,0)+(-45:1.6)$) .. controls +(135:1.6) and +(225:1.6) .. ($(0,0)+(45:1.6)$);
\draw[circ,xscale=-1] ($(0,0)+(-45:1.6)$) .. controls +(135:1.6) and +(225:1.6) .. ($(0,0)+(45:1.6)$);
\draw[circ] ($(0,0)+(45:1.6)$) .. controls +(45:1.5) and +(-45:1.5) .. ($(0,0)+(-45:1.6)$);
\draw[circ] ($(0,0)+(225:1.6)$) .. controls +(225:1) and +(180:1) .. (0,-2.5) arc (-90:90:2.5) .. controls +(180:1) and +(135:1) .. ($(0,0)+(135:1.6)$);
\end{scope}
\node at (0,1.9) {$W_{c-\epsilon}'-\iota(f^{-1}(c-\epsilon))$};
\begin{scope}[shift={(0,-2.3)}]
\coordinate (A) at (1,0);
\coordinate (Ab) at (0,0);
\coordinate (Ac) at (-1,0);
\fill (A) circle (.08);
\fill (Ab) circle (.08);
\fill (Ac) circle (.08);
\draw[conx] (Ac)--(Ab);
\draw[conx] (Ab)--(A);
\node[gnode] at (A) {$A$};
\node[gnode] at (Ab) {$B$};
\node[gnode] at (Ac) {$C$};
\node at (0,-1) {$N_{c-\epsilon}$};
\end{scope}
\begin{scope}[shift={(5,0)}]
\coordinate (b) at (0,0);
\fill[newcol] (0,0) circle (1.5);
\node[white] at (0,.8) {$B'$};
\node[white] at (-.9,0) {$A'$};
\node[white] at (.6,0) {$C'$};
\begin{scope}[scale=.5]
\draw[circ] ($(0,0)+(-45:1.6)$) -- ++(135:3.2);
\draw[circ,xscale=-1] ($(0,0)+(-45:1.6)$) -- ++(135:3.2);
\draw[circ] ($(0,0)+(45:1.6)$) .. controls +(45:1.5) and +(-45:1.5) .. ($(0,0)+(-45:1.6)$);
\draw[circ] ($(0,0)+(225:1.6)$) .. controls +(225:1) and +(180:1) .. (0,-2.5) arc (-90:90:2.5) .. controls +(180:1) and +(135:1) .. ($(0,0)+(135:1.6)$);
\end{scope}
\node at (0,1.9) {$W_{c}' - \iota(f^{-1}(c))$};
\begin{scope}[shift={(0,-2.3)}]
\coordinate (B) at (1,0);
\coordinate (Bb) at (0,0);
\coordinate (Bc) at (-1,0);
\fill (B) circle (.08);
\fill (Bb) circle (.08);
\fill (Bc) circle (.08);
\draw[conx] (Bc)--(Bb);
\draw[conx] (Bb)--(B);
\node[gnode] at (B) {$A'$};
\node[gnode] at (Bb) {$B'$};
\node[gnode] at (Bc) {$C'$};
\node at (0,-1) {$N_{c}$};
\end{scope}
\end{scope}
\begin{scope}[shift={(10,0)}]
\coordinate (c) at (0,0);
\fill[newcol] (0,0) circle (1.5);
\node[white] at (0,.8) {$B''$};
\node[white] at (-.9,0) {$A''$};
\begin{scope}[scale=.5]
\draw[circ] ($(0,0)+(-45:1.6)$) .. controls +(135:1.6) and +(45:1.6) .. ($(0,0)+(225:1.6)$);
\draw[circ,yscale=-1] ($(0,0)+(-45:1.6)$) .. controls +(135:1.6) and +(45:1.6) .. ($(0,0)+(225:1.6)$);
\draw[circ] ($(0,0)+(45:1.6)$) .. controls +(45:1.5) and +(-45:1.5) .. ($(0,0)+(-45:1.6)$);
\draw[circ] ($(0,0)+(225:1.6)$) .. controls +(225:1) and +(180:1) .. (0,-2.5) arc (-90:90:2.5) .. controls +(180:1) and +(135:1) .. ($(0,0)+(135:1.6)$);
\end{scope}
\node at (0,1.9) {$W_{c+\epsilon}' - \iota(f^{-1}(c+\epsilon))$};
\begin{scope}[shift={(0,-2.3)}]
\coordinate (bpp) at (-.7,0);
\coordinate (app) at (.7,0);
\fill (bpp) circle (.08);
\fill (app) circle (.08);
\node[gnode] at (bpp) {$B''$};
\node[gnode] at (app) {$A''$};
\draw[conx] (bpp)--(app);
\node at (0,-1) {$N_{c+\epsilon}$};
\end{scope}
\end{scope}
\draw[tmap,<-{Hooks[left]}] (a) to node[below=5pt] {$\begin{smallmatrix} A & \mapsfrom & A' \\ B & \mapsfrom & B' \\ C & \mapsfrom & C' \end{smallmatrix}$} (b);
\draw[tmap,<-{Hooks[left]}] (b) to node[below=5pt] {$\begin{smallmatrix} A' & \mapsfrom & A'' \\ B' & \mapsfrom & B'' \end{smallmatrix}$} (c);
\draw[gmap,<-{Hooks[left]}] (A) to node[above] {$\cong$} (Bc);
\draw[gmap,<-{Hooks[left]}] (B) -- (bpp);
\end{tikzpicture}
\caption{Topological construction and labelling for Statement~\ref{thm:nesting5}.}
\label{fig:nesting4}
\end{figure}

For Statement~\ref{thm:nesting5}, as in Figure \ref{fig:nesting4}, assign labels to the connected components of the (subsets of) level sets $W_{c-\epsilon}'-\iota(f^{-1}(c-\epsilon))$, $W_c'-\iota(f^{-1}(c))$, and $W_{c+\epsilon}'-\iota(f^{-1}(c+\epsilon))$. The topological inclusions $A'\hookrightarrow A$, $B'\xhookrightarrow{\cong} B$, and $C'\xhookrightarrow{\cong} C$ induce an analogous nesting poset $N_c\to N_{c-\epsilon}$. Similarly, the topological inclusions $A''\xhookrightarrow{\cong}A'$ and $B''\xhookrightarrow{\cong}B'$ induce an analogous nesting poset map $N_{c+\epsilon}\to N_c$.
\qed
\end{proof}

Note that the maps in Statements~\ref{thm:nesting1}~-~\ref{thm:nesting4}
all have a zigzag structure $N_{c-\epsilon} \leftarrow N_c \rightarrow N_{c+\epsilon}$, but Statements ~\ref{thm:nesting5} and~\ref{thm:nesting6} do not. One may enforce such maps on Statements~\ref{thm:nesting5} and~\ref{thm:nesting6}, and following Figure \ref{fig:nesting4} we have two choices. One choice is to map $C'$ and $A'$ both to $A''$, following the induced topology, but that would collapse the poset down to a single point, as we require order-preservation. The second choice is to map $C'$ and $B'$ both to $B''$, but that would break functoriality and not follow the topology of the neighborhood.

\begin{remark}
We mention some observations from this section so far.
\begin{enumerate}
\item The nesting poset doesn't see the critical value in $[c-\epsilon,c]$ if $c$ is merging or a minimum, because the interiors have not merged yet.
\item Functoriality seems to hold from the category of topological spaces $\mathbf{Top}$ to the category of posets and order-preserving maps. However, we only assigned poset maps to particular topological inclusions, not all of them, as mentioned in the comment after the proof of Theorem \ref{theorem:nesting-poset-B}, so functoriality may hold for an appropriately defined subcategory of $\mathbf{Top}$.
\item In Figures \ref{fig:nesting1}, \ref{fig:nesting3}, \ref{fig:nesting4} there was always a largest poset element, and in fact $N_t$ always has a maximal element representing the unbounded component of $\pi^{-1}(t)-\iota(f^{-1}(t))$. The poset maps always send the maximal element to the maximal element, so it does not contain any interesting information.
\end{enumerate}
\end{remark}

For completeness, we include in Figure \ref{fig:nesting-ALL} an exhaustive description, up to diffeomorphism, of all of the choices presented by the ambiguous connected component labelling from Figure \ref{fig:nesting2}.

\begin{figure}[!ht]\centering
\newcommand\labelleft{-6.5} 
\begin{tikzpicture}[scale=.6,circ/.style={line width=2pt,draw=white}]
\node at (\labelleft,1.2) {(a) non-nesting};
\node at (\labelleft,.4) {merging saddle,};
\node at (\labelleft,-.4) {corresponding to};
\node at (\labelleft,-1.2) {statement~\ref{thm:nesting3}};
\node at (0,4) {$W_{c-\epsilon}'-\iota(f^{-1}(c-\epsilon))$};
\fill[newcol,opacity=.4] (0,0) circle (3);
\fill[newcol] (0,0) circle (1.5);
\draw[circ] ($(0,0)+(-45:1.6)$) .. controls +(135:1.6) and +(225:1.6) .. ($(0,0)+(45:1.6)$);
\draw[circ,xscale=-1] ($(0,0)+(-45:1.6)$) .. controls +(135:1.6) and +(225:1.6) .. ($(0,0)+(45:1.6)$);
\draw[circ] ($(0,0)+(45:1.6)$) .. controls +(45:2) and +(-45:2) .. ($(0,0)+(-45:1.6)$);
\draw[circ,xscale=-1] ($(0,0)+(45:1.6)$) .. controls +(45:2) and +(-45:2) .. ($(0,0)+(-45:1.6)$);
\begin{scope}[shift={(7,0)}]
\node at (0,4) {$W_{c}'-\iota(f^{-1}(c))$};
\fill[newcol,opacity=.4] (0,0) circle (3);
\fill[newcol] (0,0) circle (1.5);
\draw[circ] ($(0,0)+(-45:1.6)$) -- ++(135:3.2);
\draw[circ,xscale=-1] ($(0,0)+(-45:1.6)$) -- ++(135:3.2);
\draw[circ] ($(0,0)+(45:1.6)$) .. controls +(45:2) and +(-45:2) .. ($(0,0)+(-45:1.6)$);
\draw[circ,xscale=-1] ($(0,0)+(45:1.6)$) .. controls +(45:2) and +(-45:2) .. ($(0,0)+(-45:1.6)$);
\end{scope}
\begin{scope}[shift={(14,0)}]
\node at (0,4) {$W_{c+\epsilon}'-\iota(f^{-1}(c+\epsilon))$};
\fill[newcol,opacity=.4] (0,0) circle (3);
\fill[newcol] (0,0) circle (1.5);
\draw[circ] ($(0,0)+(-45:1.6)$) .. controls +(135:1.6) and +(45:1.6) .. ($(0,0)+(225:1.6)$);
\draw[circ,yscale=-1] ($(0,0)+(-45:1.6)$) .. controls +(135:1.6) and +(45:1.6) .. ($(0,0)+(225:1.6)$);
\draw[circ] ($(0,0)+(45:1.6)$) .. controls +(45:2) and +(-45:2) .. ($(0,0)+(-45:1.6)$);
\draw[circ,xscale=-1] ($(0,0)+(45:1.6)$) .. controls +(45:2) and +(-45:2) .. ($(0,0)+(-45:1.6)$);
\end{scope}
\begin{scope}[shift={(0,-7)}]
\node at (\labelleft,1.2) {(b) non-nesting};
\node at (\labelleft,.4) {splitting saddle,};
\node at (\labelleft,-.4) {corresponding to};
\node at (\labelleft,-1.2) {statement~\ref{thm:nesting4}};
\fill[newcol,opacity=.4] (0,0) circle (3);
\fill[newcol] (0,0) circle (1.5);
\draw[circ] ($(0,0)+(-45:1.6)$) .. controls +(135:1.6) and +(225:1.6) .. ($(0,0)+(45:1.6)$);
\draw[circ,xscale=-1] ($(0,0)+(-45:1.6)$) .. controls +(135:1.6) and +(225:1.6) .. ($(0,0)+(45:1.6)$);
\draw[circ] ($(0,0)+(45:1.6)$) .. controls +(45:2) and +(135:2) .. ($(0,0)+(135:1.6)$);
\draw[circ,yscale=-1] ($(0,0)+(45:1.6)$) .. controls +(45:2) and +(135:2) .. ($(0,0)+(135:1.6)$);
\begin{scope}[shift={(7,0)}]
\fill[newcol,opacity=.4] (0,0) circle (3);
\fill[newcol] (0,0) circle (1.5);
\draw[circ] ($(0,0)+(-45:1.6)$) -- ++(135:3.2);
\draw[circ,xscale=-1] ($(0,0)+(-45:1.6)$) -- ++(135:3.2);
\draw[circ] ($(0,0)+(45:1.6)$) .. controls +(45:2) and +(135:2) .. ($(0,0)+(135:1.6)$);
\draw[circ,yscale=-1] ($(0,0)+(45:1.6)$) .. controls +(45:2) and +(135:2) .. ($(0,0)+(135:1.6)$);
\end{scope}
\begin{scope}[shift={(14,0)}]
\fill[newcol,opacity=.4] (0,0) circle (3);
\fill[newcol] (0,0) circle (1.5);
\draw[circ] ($(0,0)+(-45:1.6)$) .. controls +(135:1.6) and +(45:1.6) .. ($(0,0)+(225:1.6)$);
\draw[circ,yscale=-1] ($(0,0)+(-45:1.6)$) .. controls +(135:1.6) and +(45:1.6) .. ($(0,0)+(225:1.6)$);
\draw[circ] ($(0,0)+(45:1.6)$) .. controls +(45:2) and +(135:2) .. ($(0,0)+(135:1.6)$);
\draw[circ,yscale=-1] ($(0,0)+(45:1.6)$) .. controls +(45:2) and +(135:2) .. ($(0,0)+(135:1.6)$);
\end{scope}
\end{scope}
\begin{scope}[shift={(0,-14)}]
\node at (\labelleft,1.2) {(c) nesting};
\node at (\labelleft,.4) {merging saddle,};
\node at (\labelleft,-.4) {corresponding to};
\node at (\labelleft,-1.2) {statement~\ref{thm:nesting5}};
\fill[newcol,opacity=.4] (0,0) circle (3);
\fill[newcol] (0,0) circle (1.5);
\draw[circ] ($(0,0)+(-45:1.6)$) .. controls +(135:1.6) and +(225:1.6) .. ($(0,0)+(45:1.6)$);
\draw[circ,xscale=-1] ($(0,0)+(-45:1.6)$) .. controls +(135:1.6) and +(225:1.6) .. ($(0,0)+(45:1.6)$);
\draw[circ] ($(0,0)+(45:1.6)$) .. controls +(45:1.5) and +(-45:1.5) .. ($(0,0)+(-45:1.6)$);
\draw[circ] ($(0,0)+(225:1.6)$) .. controls +(225:1) and +(180:1) .. (0,-2.5) arc (-90:90:2.5) .. controls +(180:1) and +(135:1) .. ($(0,0)+(135:1.6)$);
\begin{scope}[shift={(7,0)}]
\fill[newcol,opacity=.4] (0,0) circle (3);
\fill[newcol] (0,0) circle (1.5);
\draw[circ] ($(0,0)+(-45:1.6)$) -- ++(135:3.2);
\draw[circ,xscale=-1] ($(0,0)+(-45:1.6)$) -- ++(135:3.2);
\draw[circ] ($(0,0)+(45:1.6)$) .. controls +(45:1.5) and +(-45:1.5) .. ($(0,0)+(-45:1.6)$);
\draw[circ] ($(0,0)+(225:1.6)$) .. controls +(225:1) and +(180:1) .. (0,-2.5) arc (-90:90:2.5) .. controls +(180:1) and +(135:1) .. ($(0,0)+(135:1.6)$);
\end{scope}
\begin{scope}[shift={(14,0)}]
\fill[newcol,opacity=.4] (0,0) circle (3);
\fill[newcol] (0,0) circle (1.5);
\draw[circ] ($(0,0)+(-45:1.6)$) .. controls +(135:1.6) and +(45:1.6) .. ($(0,0)+(225:1.6)$);
\draw[circ,yscale=-1] ($(0,0)+(-45:1.6)$) .. controls +(135:1.6) and +(45:1.6) .. ($(0,0)+(225:1.6)$);
\draw[circ] ($(0,0)+(45:1.6)$) .. controls +(45:1.5) and +(-45:1.5) .. ($(0,0)+(-45:1.6)$);
\draw[circ] ($(0,0)+(225:1.6)$) .. controls +(225:1) and +(180:1) .. (0,-2.5) arc (-90:90:2.5) .. controls +(180:1) and +(135:1) .. ($(0,0)+(135:1.6)$);
\end{scope}
\end{scope}
\begin{scope}[shift={(0,-21)}]
\node at (\labelleft,1.2) {(d) nesting};
\node at (\labelleft,.4) {splitting saddle,};
\node at (\labelleft,-.4) {corresponding to};
\node at (\labelleft,-1.2) {statement~\ref{thm:nesting6}};
\fill[newcol,opacity=.4] (0,0) circle (3);
\fill[newcol] (0,0) circle (1.5);
\draw[circ] ($(0,0)+(-45:1.6)$) .. controls +(135:1.6) and +(225:1.6) .. ($(0,0)+(45:1.6)$);
\draw[circ,xscale=-1] ($(0,0)+(-45:1.6)$) .. controls +(135:1.6) and +(225:1.6) .. ($(0,0)+(45:1.6)$);
\draw[circ] ($(0,0)+(45:1.6)$) .. controls +(45:1.5) and +(135:1.5) .. ($(0,0)+(135:1.6)$);
\draw[circ] ($(0,0)+(-45:1.6)$) .. controls +(-45:1) and +(270:1) .. (2.5,0) arc (0:180:2.5) .. controls +(270:1) and +(225:1) .. ($(0,0)+(225:1.6)$);
\begin{scope}[shift={(7,0)}]
\fill[newcol,opacity=.4] (0,0) circle (3);
\fill[newcol] (0,0) circle (1.5);
\draw[circ] ($(0,0)+(-45:1.6)$) -- ++(135:3.2);
\draw[circ,xscale=-1] ($(0,0)+(-45:1.6)$) -- ++(135:3.2);
\draw[circ] ($(0,0)+(45:1.6)$) .. controls +(45:1.5) and +(135:1.5) .. ($(0,0)+(135:1.6)$);
\draw[circ] ($(0,0)+(-45:1.6)$) .. controls +(-45:1) and +(270:1) .. (2.5,0) arc (0:180:2.5) .. controls +(270:1) and +(225:1) .. ($(0,0)+(225:1.6)$);
\end{scope}
\begin{scope}[shift={(14,0)}]
\fill[newcol,opacity=.4] (0,0) circle (3);
\fill[newcol] (0,0) circle (1.5);
\draw[circ] ($(0,0)+(-45:1.6)$) .. controls +(135:1.6) and +(45:1.6) .. ($(0,0)+(225:1.6)$);
\draw[circ,yscale=-1] ($(0,0)+(-45:1.6)$) .. controls +(135:1.6) and +(45:1.6) .. ($(0,0)+(225:1.6)$);
\draw[circ] ($(0,0)+(45:1.6)$) .. controls +(45:1.5) and +(135:1.5) .. ($(0,0)+(135:1.6)$);
\draw[circ] ($(0,0)+(-45:1.6)$) .. controls +(-45:1) and +(270:1) .. (2.5,0) arc (0:180:2.5) .. controls +(270:1) and +(225:1) .. ($(0,0)+(225:1.6)$);
\end{scope}
\end{scope}
\end{tikzpicture}
\caption{Larger neighborhoods $W'\supseteq W$ corresponding to Statements~\ref{thm:nesting3}-\ref{thm:nesting6} of Theorem \ref{theorem:nesting-poset-B}.}
\label{fig:nesting-ALL}
\end{figure}
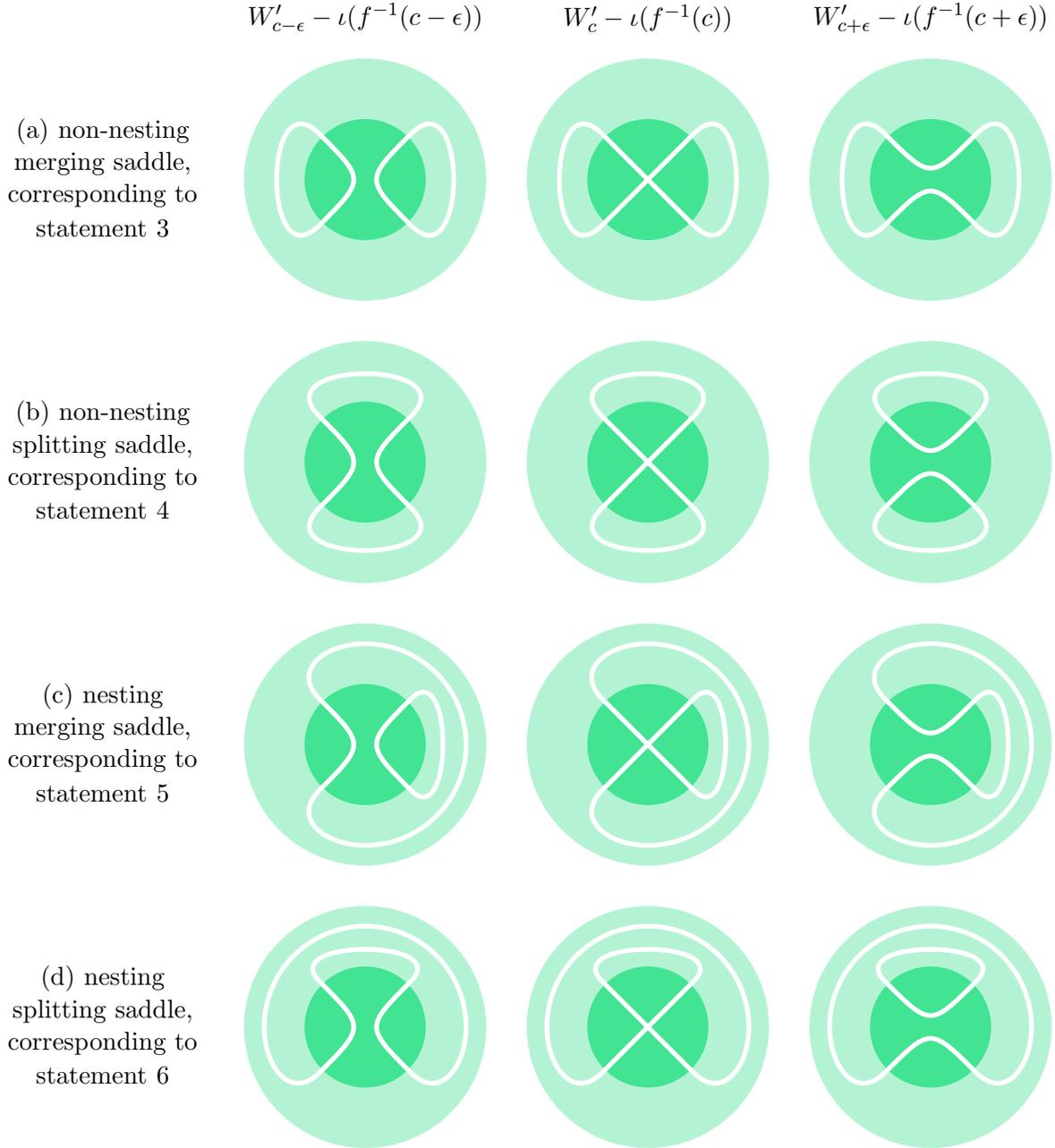

\subsection{The Zigzag of Posets}

\subsubsection{Combinatorial Barcode}
We now have a zigzag structure along $\Rspace$ of the nesting
posets of $f$. This new data will allow us to augment the data of the barcode with a new type of barcode that combinatorially describes a Morse function factoring through an embedding by a height projection.

\begin{corollary}
\label{cor:zigzag}
Given a slicing $a_0 < t_1 < a_1 < \cdots < t_n < a_{n}$, there is a zigzag of posets
\begin{equation}
\label{eqn:nesting-zigzag}
N_{a_0} \xleftrightarrow{\varphi^-_0} N_{t_1} \xleftrightarrow{\varphi^+_0} N_{a_1} \xleftrightarrow{\varphi^-_1} N_{t_2} \xleftrightarrow{\varphi^+_1} \cdots \xleftrightarrow{\varphi^-_n} N_{t_n} \xleftrightarrow{\varphi^+_n} N_{a_{n}},
\end{equation}
where the direction of each $\varphi^\pm_i$ is defined by Theorem \ref{theorem:nesting-poset-B}.
\end{corollary}

This follows directly from Theorem \ref{theorem:nesting-poset-B}. Note that every backwards map $\varphi$  in the zigzag \eqref{eqn:nesting-zigzag} can be reversed using the Galois connection construction \cite[Chapter 3]{roman2008lattices} with
\begin{equation}
\varphi^{\dagger}(y) \colonequals \max \{x:~\varphi(x) \leq y\},
\end{equation}
though we are only guaranteed $\varphi^{\dagger}(\varphi(x)) \geq x$. That is, while we do get a diagram
\begin{equation}
\label{diag:straightened-zigzag}
N_{a_0} \longrightarrow N_{t_1} \longrightarrow N_{a_1} \longrightarrow \cdots \longrightarrow N_{t_n} \longrightarrow N_{a_n}
\end{equation}
of posets, it does not capture splitting saddles (neither nesting nor non-nesting). Hence we instead explore a combinatorial barcode with $\field$-algebras corresponding to posets in a persistence module.

\begin{remark}
The zigzag \eqref{eqn:nesting-zigzag} is reminiscent of diagrams in zigzag persistence \cite{zigzag}, and at first glance it seems possible to recover \eqref{eqn:nesting-zigzag} by taking the nesting posets of spaces whose homology is taken to compute zigzag persistence. However, in zigzag persistence the considered spaces are of the sort $f^{-1}[t,s]$, whereas in our case we compute the nesting poset of (a subset of) $f^{-1}(t)$. For such $t$, there does not always exist $\epsilon>0$ such that $f^{-1}(t\pm \epsilon)$ is the same nesting poset as for $f^{-1}(t)$, and so computing the nesting poset of an interlevel set does not make sense in our context. Nonetheless, there are modifications \cite{kimmemoli} of this approach that hold promise for applications.
\end{remark}

For a poset $P$, an \emph{interval} $I$ in $P$ is a connected subposet $I\subseteq P$ such that for any $x, y \in I$, $x \leq t \leq y$ implies $t \in I$. When the poset if finite, an interval $I$ is generated by two endpoints, and we write $I = \left[ a, b \right] \colonequals \{t\in P :~ a \leq t \leq b \}$. The following definition comes from \cite{charalambides2018enumerative}.

\begin{definition}
The \textit{incidence algebra} $\field  P $ of a poset $P$ is the free vector space over $\field$ generated by set of intervals $I$ of $P$. Multiplication $\times \colon \field  P  \times \field  P  \rightarrow \field  P$ is given by concatenation of compatible intervals
\[
\left[ c,d \right] \times \left[ a,b \right] =
\begin{cases}
\left[ a,d \right] & \text{if\ }b=c \\
0 & \text{else,}
\end{cases}
\]
and multiplication is $0$ for unconcatenable intervals, making $\field  P $ a $\field$-algebra.
\end{definition}

\begin{conjecture}
\label{conj:zigzag-algebra}
There is a zigzag module of $\field$-algebras
\[
\field N_{a_0} \leftrightarrow \field N_{t_1} \leftrightarrow \field N_{a_1} \leftrightarrow \field N_{t_2} \leftrightarrow \cdots \leftrightarrow \field N_{t_n} \leftrightarrow \field N_{a_{n}},
\]
with arrow direction given by Corollary \ref{cor:zigzag}, that decomposes into a sum of interval indecomposables of the form
\[
\field_I (t) =
\begin{cases}
\field \{*\} & \text{if\ } t \in I \\
0 & \text{if\ } t \notin I.
\end{cases}
\]
Furthermore, this collection of interval indecomposables determines $f$ up to poset equivalence.
\end{conjecture}

This approach follows the vein of foundational persistent homology results \cite{Crawley-Boevey2015,Botnan2017} about decomposition of barcodes into fundamental parts. We call this collection of interval indecomposables the \emph{combinatorial barcode} of $f$.

\section{Counting Morse Functions}
\label{sec:results-counting}

The observations of the previous section, specifically Figure \ref{fig:crit-types},  hint to a method of counting Morse functions by their barcode. In this section, instead of analyzing the local behavior around critical values as before, we start with a global picture of a complete barcode, and use Figure \ref{fig:counting-example} as motivation. As before, $f\colon \Sspace^2 \rightarrow \R$ factors as $\pi\circ\iota$, for $\iota\colon \Sspace^2\to \Rspace^3$ a smooth embedding.

\begin{figure}[h]\centering
\begin{tikzpicture}[scale=.6]
\begin{scope}[shift={(-27*\bwid,0)}]
\node (lab) at (-2,1) {(a) Given};
\node[anchor=north] at (lab.south) {barcode};
\draw[barco] (.3,0) -- (.3,1);
\draw[barco] (.6,-1) -- (.6,2);
\draw[barcc] (.9,-2) -- (.9,3);
\end{scope}
\begin{scope}[shift={(-\bwid,0)}]
\begin{scope}[shift={(-2*\bwid,0)}]
\draw[barcc] (.3,-2) -- (.3,3);
\foreach \y in {-2,...,3}{
  \draw[opacity=.5,dotted] (0,\y)--(.6*\hfactor*\bwid,\y);
}
\end{scope}
\begin{scope}[shift={(1.5*\bwid,0)}]
\node at (0,-2.5) {$\iota_1$};
\node (lab) at (-3.9,1) {(b) 1st bar: one};
\node[anchor=north] at (lab.south) {embedding};
\draw[fill=newcol] (0,3) arc (90:0:\bwid/2) --++ (270:5-\bwid)
  arc (0:-180:\bwid/2) --++ (90:5-\bwid)
  arc (180:90:\bwid/2);
\end{scope}
\end{scope}
\begin{scope}[shift={(3*\hfactor*\bwid,0)}]
\begin{scope}[shift={(-1.4*\bwid,0)}]
\draw[barco] (.3,-1) -- (.3,2);
\draw[barcc] (.6,-2) -- (.6,3);
\foreach \y in {-2,...,3}{
  \draw[opacity=.5,dotted] (0,\y)--(13*\bwid,\y);
}
\node (lab) at (-2.8,1) {(c) 2nd bar: two};
\node[anchor=north] at (lab.south) {embeddings};
\end{scope}
\begin{scope}[shift={(4.5*\bwid,0)}]
\node at (-1*\bwid,-2.5) {$\iota_{11}$};
\draw[fill=newcol] (0,3) arc (90:0:\bwid/2) --++ (270:5-\bwid)
  arc (0:-180:\bwid/2) --++ (90:4-\bwid)
  arc (0:180:\bwid/2) --++ (270:3-\bwid)
  arc (0:-180:\bwid/2) --++ (90:3-\bwid/2)
  arc (180:90:\bwid) arc (-90:0:\bwid) --++ (90:1-2.5*\bwid)
  arc (180:90:\bwid/2);
\end{scope}
\begin{scope}[shift={(9.5*\bwid,0)}]
\node at (-1*\bwid,-2.5) {$\iota_{12}$};
\draw[fill=newcol] (0,3) arc (90:0:\bwid/2) --++ (270:5-2*\bwid)
  arc (0:-180:1.5*\bwid) --++ (90:4-2*\bwid)
  arc (180:90:\bwid/2) arc (-90:0:1.5*\bwid) --++ (90:1-2*\bwid)
  arc (180:90:\bwid/2);
\draw[fill=newcol!66,dashed] (-2*\bwid,2) arc (90:0:\bwid/2) --++ (270:3-\bwid)
  arc (-180:0:\bwid/2) --++ (90:3+\bwid)
  arc (0:-90:1.5*\bwid);
\end{scope}
\end{scope}
\begin{scope}[shift={(-3,-6.5)}]
\begin{scope}[shift={(-2*\bwid,0)}]
\draw[barco] (.3,0) -- (.3,1);
\draw[barco] (.6,-1) -- (.6,2);
\draw[barcc] (.9,-2) -- (.9,3);
\foreach \y in {-2,...,3}{
  \draw[opacity=.5,dotted] (0,\y)--(32*\bwid,\y);
}
\node (lab) at (-3.8,1) {(d) 3rd bar: four};
\node[anchor=north] at (lab.south) {embeddings from $\iota_{11}$};
\end{scope}
\begin{scope}[shift={(7.5*\bwid,0)}]
\node at (-2*\bwid,-2.5) {$\iota_{111}$};
\draw[fill=newcol] (0,3) arc (90:0:\bwid/2) --++ (270:5-\bwid)
  arc (0:-180:\bwid/2) --++ (90:4-\bwid)
  arc (0:180:\bwid/2) --++ (270:3-\bwid)
  arc (0:-180:\bwid/2) --++ (90:2-\bwid)
  arc (0:180:\bwid/2) --++ (270:1-\bwid)
  arc (0:-180:\bwid/2) --++ (90:1-\bwid/2)
  arc (180:90:\bwid) arc (-90:0:\bwid) --++ (90:1-2*\bwid)
  arc (180:90:\bwid) arc (-90:0:\bwid) --++ (90:1-2.5*\bwid)
  arc (180:90:\bwid/2);
\end{scope}
\begin{scope}[shift={(14.5*\bwid,0)}]
\node at (-2*\bwid,-2.5) {$\iota_{112}$};
\draw[fill=newcol] (0,3) arc (90:0:\bwid/2) --++ (270:5-\bwid)
  arc (0:-180:\bwid/2) --++ (90:4-\bwid)
  arc (0:180:\bwid/2) --++ (270:3-2*\bwid)
  arc (0:-180:1.5*\bwid) --++ (90:2-2*\bwid)
  arc (180:90:\bwid/2) arc (-90:0:1.5*\bwid) --++ (90:1-1.5*\bwid)
  arc (180:90:\bwid) arc (-90:0:\bwid) --++ (90:1-2.5*\bwid)
  arc (180:90:\bwid/2);
\draw[fill=newcol!66,dashed] (-4*\bwid,1) arc (90:0:\bwid/2) --++ (270:1-\bwid)
  arc (180:360:\bwid/2) --++ (90:1+\bwid)
  arc (0:-90:1.5*\bwid);
\end{scope}
\begin{scope}[shift={(21.5*\bwid,0)}]
\node at (-2*\bwid,-2.5) {$\iota_{113}$};
\draw[fill=newcol] (-2*\bwid,3) arc (90:0:\bwid/2) --++ (270:2-2.5*\bwid)
  arc (180:270:\bwid) arc (90:0:\bwid) --++ (270:1-\bwid/2)
  arc (0:-180:\bwid/2) --++ (90:1-\bwid)
  arc (0:180:\bwid/2) --++ (270:3-\bwid)
  arc (0:-180:\bwid/2) --++ (90:4-\bwid)
  arc (0:180:\bwid/2) --++ (270:3-\bwid)
  arc (0:-180:\bwid/2) --++ (90:3-\bwid/2)
  arc (180:90:\bwid) arc (-90:0:\bwid) --++ (90:1-2.5*\bwid)
  arc (180:90:\bwid/2);
\end{scope}
\begin{scope}[shift={(28.5*\bwid,0)}]
\node at (-2*\bwid,-2.5) {$\iota_{114}$};
\draw[fill=newcol] (-2*\bwid,3) arc (90:0:\bwid/2) --++ (270:2-2*\bwid)
  arc (180:270:1.5*\bwid) arc (90:0:\bwid/2) --++ (270:3-2*\bwid)
  arc (0:-180:1.5*\bwid) --++ (90:4-2*\bwid)
  arc (0:180:\bwid/2) --++ (270:3-\bwid)
  arc (0:-180:\bwid/2) --++ (90:3-\bwid/2)
  arc (180:90:\bwid) arc (-90:0:\bwid) --++ (90:1-2.5*\bwid)
  arc (180:90:\bwid/2);
\draw[fill=newcol!66,dashed] (0,1) arc (90:180:\bwid/2) --++ (270:1-\bwid)
  arc (0:-180:\bwid/2) --++ (90:1+\bwid)
  arc (-180:-90:1.5*\bwid);
\end{scope}
\end{scope}
\begin{scope}[shift={(-3,-13)}]
\begin{scope}[shift={(-2*\bwid,0)}]
\draw[barco] (.3,0) -- (.3,1);
\draw[barco] (.6,-1) -- (.6,2);
\draw[barcc] (.9,-2) -- (.9,3);
\foreach \y in {-2,...,3}{
  \draw[opacity=.5,dotted] (0,\y)--(32*\bwid,\y);
}
\node (lab) at (-3.8,1) {(e) 3rd bar: four};
\node[anchor=north] at (lab.south) {embeddings from $\iota_{12}$};
\end{scope}
\begin{scope}[shift={(7.5*\bwid,0)}]
\node at (-2*\bwid,-2.5) {$\iota_{121}$};
\draw[fill=newcol] (0,3) arc (90:0:\bwid/2) --++ (270:5-2*\bwid)
  arc (0:-180:1.5*\bwid) --++ (90:3-2*\bwid)
  arc (0:180:\bwid/2) --++ (270:1-\bwid)
  arc (0:-180:\bwid/2) --++ (90:1-\bwid/2)
  arc (180:90:\bwid) arc (-90:0:\bwid) --++ (90:1-2.5*\bwid)
  arc (180:90:\bwid/2) arc (-90:0:1.5*\bwid) --++ (90:1-2*\bwid)
  arc (180:90:\bwid/2);
\draw[fill=newcol!66,dashed] (-2*\bwid,2) arc (90:0:\bwid/2) --++ (270:3-\bwid)
  arc (-180:0:\bwid/2) --++ (90:3+\bwid)
  arc (0:-90:1.5*\bwid);
\end{scope}
\begin{scope}[shift={(14.5*\bwid,0)}]
\node at (-2*\bwid,-2.5) {$\iota_{122}$};
\draw[fill=newcol] (0,3) arc (90:0:\bwid/2) --++ (270:5-3*\bwid)
  arc (0:-180:2.5*\bwid) --++ (90:3-3*\bwid)
  arc (180:90:\bwid/2) arc (-90:0:1.5*\bwid) --++ (90:1-2*\bwid)
  arc (180:90:\bwid/2) arc (-90:0:1.5*\bwid) --++ (90:1-2*\bwid)
  arc (180:90:\bwid/2);
\draw[fill=newcol!66,dashed] (-4*\bwid,1) arc (90:0:\bwid/2) --++ (270:1-\bwid)
  arc (180:360:\bwid/2) --++ (90:1+\bwid)
  arc (0:-90:1.5*\bwid);
\draw[fill=newcol!66,dashed] (-2*\bwid,2) arc (90:0:\bwid/2) --++ (270:3-\bwid)
  arc (-180:0:\bwid/2) --++ (90:3+\bwid)
  arc (0:-90:1.5*\bwid);
\end{scope}
\begin{scope}[shift={(21.5*\bwid,0)}]
\node at (-2*\bwid,-2.5) {$\iota_{123}$};
\draw[fill=newcol] (0,3) arc (90:0:\bwid/2) --++ (270:5-3*\bwid)
  arc (0:-180:2.5*\bwid) --++ (90:4-5*\bwid)
  arc (180:90:2.5*\bwid) arc (-90:0:1.5*\bwid) --++ (90:1-2*\bwid)
  arc (180:90:\bwid/2);
\draw[fill=newcol!66,dashed] (-2*\bwid,2) arc (90:0:\bwid/2) --++ (270:1-2.5*\bwid)
  arc (0:-90:\bwid) arc (90:180:\bwid) --++ (270:1-\bwid/2)
  arc (-180:0:\bwid/2) --++ (90:1-\bwid)
  arc (180:0:\bwid/2) --++ (270:2-\bwid)
  arc (-180:0:\bwid/2) --++ (90:3+\bwid)
  arc (0:-90:1.5*\bwid);
\end{scope}
\begin{scope}[shift={(28.5*\bwid,0)}]
\node at (-2*\bwid,-2.5) {$\iota_{124}$};
\draw[fill=newcol] (0,3) arc (90:0:\bwid/2) --++ (270:5-3*\bwid)
  arc (0:-180:2.5*\bwid) --++ (90:4-5*\bwid)
  arc (180:90:2.5*\bwid) arc (-90:0:1.5*\bwid) --++ (90:1-2*\bwid)
  arc (180:90:\bwid/2);
\draw[fill=newcol!66,dashed] (-2*\bwid,2) arc (90:0:\bwid/2) --++ (270:1-2*\bwid)
  arc (0:-90:1.5*\bwid) arc (90:180:\bwid/2) --++ (270:2-2*\bwid)
  arc (-180:0:1.5*\bwid) --++ (90:3)
  arc (0:-90:1.5*\bwid);
\draw[fill=newcol!33,dashed] (-3*\bwid,1) arc (90:0:\bwid/2) --++ (270:1-\bwid)
  arc (180:360:\bwid/2) --++ (90:1+\bwid)
  arc (0:-90:1.5*\bwid);
\end{scope}
\end{scope}
\end{tikzpicture}
\caption{A motivating example.}
\label{fig:counting-example}
\end{figure}
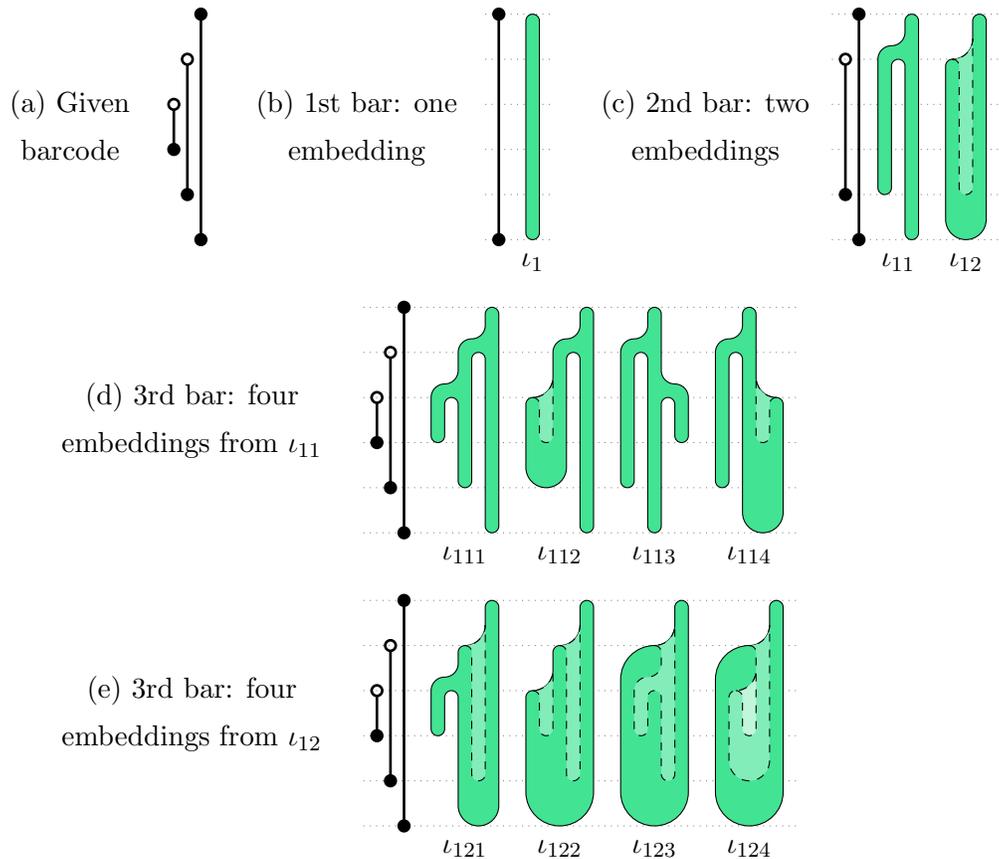

\begin{example}\label{ex:one-max}
Suppose $f$ has $6$ critical values and a known zero-dimensional barcode consisting of $3$ bars nested inside each other, as in Figure \ref{fig:counting-example}(a). Construct all embeddings $\iota$ of $f$ by considering the effect of each bar on the embedding separately, following the nesting/non-nesting poset approach of Section \ref{sec:results-posets}. Begin with the largest bar of the barcode, then add bars in their nesting order. The second bar gives $2$ nesting choices. Each of the $2$ nesting choices gives $4$ more nesting choices, when the smallest bar of the barcode is added.
\end{example}

Example \ref{ex:one-max} leads immediately to several observations. First note that the types of critical points associated with closed endpoints of barcodes are decided (local minimum at the highest points, local maximum at the lowest points). Second, we see that simply choosing nesting or non-nesting-type at the open barcode endpoints does not give all the embeddings.
For example, $\iota_{111}$ and $\iota_{113}$ have identical critical value types in the same order. Finally, note that the relation between the number of bars and number of embeddings depends upon containment relations among the bars. That is, by considering bars largest to smallest, for every bar contained in a larger one, the number of embeddings computed up to that point doubles. This is  more precisely described by Conjecture \ref{conj:counting}.

Let $B$ be the barcode of $f$, viewed as a set of subintervals $I_1,\dots,I_N$ of $\Rspace$. For every $j=1,\dots,N$, let $\mu_B(I_j)$ be the number of bars $I_k$ in $B$ such that $I_j\subsetneq I_k$.

\begin{conjecture}
\label{conj:counting}
The number of ways the Morse function $f$ factors through $\Rspace^3$, up to height-equivalence, is bounded below by
\begin{equation}
\label{eqn:count}
2^{N-1}\prod_{j=2}^N \mu_B(I_j).
\end{equation}
\end{conjecture}

We leave this conjecture open for future work, and make two observations about why Example \ref{ex:one-max} is not a generic example.
\begin{itemize}
\item The barcode in the example has a single interval whose highest endpoint is closed. Given more than one such interval, the count given in Conjecture \ref{conj:counting} would miss such embeddings.
\item The barcode in the example does not give rise to embeddings whose branches might be ``twisted" in a non-trivial manner. That is, given a barcode with more bars, the count from Conjecture \ref{conj:counting} would miss the embeddings with ``twists".
\end{itemize}

\section{Discussions}
\label{sec:discussions}

We have characterized the moduli space of classes of Morse functions on the sphere under both functional and dynamic settings. Using persistence as a constraint, we have defined equivalence relations between Morse functions and have studied the combinatorial structure of Morse functions on a sphere modulo such relations. Our approach describes structures in detail, and therefore provides a fruitful ground for continued research in several directions.

\subsection{Realizing Preimages of a Barcode}

Conjecture \ref{conj:counting} considers a counting argument that is only at the
beginning of addressing Objective~3 from Section \ref{sec:introduction}.  This
is given in the context of finding a representative of the preimage of the
persistence map
\begin{equation}
\{ \text{Morse functions on } \Sspace^2\} \rightarrow \{ \text{Barcodes} \},
\end{equation}
where a barcode is viewed as a multiset of intervals of $\Rspace$. To further
answer Objective 3, we may ask: given a barcode $B$, find an embedding $\iota
\colon \Sspace^2 \rightarrow \Rspace^3$ such that $B$ is the barcode of $f
\colonequals \pi \circ \iota$ under level set persistence in degree zero. Note
that some barcodes cannot be realized as height-embedded Morse functions on
$\Sspace^2$, for example $B = \{ \left[ 0, 3 \right], \left( 1, 2 \right) \}$,
as any open interval or a closed interval nested in another closed interval is
forbidden for the sphere. However, if $B$ has a single closed bar inside which
all other bars are contained, we can easily construct a Reeb graph
$\mathcal{R}_B$, which in turn may be associated to a diagram of 1-spheres and
wedges of 1-spheres, which may then be assembled into an embedding of a
2-sphere, as in Figure \ref{fig:preimage-realization}.

\begin{figure}[h]\centering
\newcommand\circhy{.06}
\begin{tikzpicture}[scale=1]
\draw[barcc] (-.3,0)--(-.3,5);
\draw[barco] (0,1)--(0,4);
\draw[barco] (.3,2)--(.3,3);
\node at (0,-1) {barcode\vphantom{graph}};
\begin{scope}[shift={(2.5,0)}]
\draw[line width=1pt] (-.75,0)--(-.75,5);
\draw[line width=1pt] (-.75,4) to [out=0,in=90] (0,3)--(0,1);
\draw[line width=1pt] (0,3) to [out=0,in=90] (.75,2);
\node at (0,-1) {Reeb graph};
\end{scope}
\begin{scope}[shift={(6,0)}]
\node at (0,-1) {diagram in $\mathbf{Top}$};
\begin{scope}[xscale=1.5,yscale=2]
\coordinate (a0) at (-1,0); \coordinate (a1) at (-1,2); \coordinate (a2) at (-1,3);
\coordinate (b0) at (0,0); \coordinate (b1) at (0,1);
\coordinate (c0) at (1,0);
\foreach \coord\lab in {a0/{*}, a1/{\Sspace^1\vee \Sspace^1}, a2/{*}, b0/{*}, b1/{\Sspace^1\vee \Sspace^1}, c0/{*}}{
  \node (\coord) at (\coord) {$\lab$};
}
\foreach \top\bot in {a0/a1, b0/b1, a1/a2, b1/a1, c0/b1}{
  \node (mid) at ($(\bot)!.5!(\top)$) {$\Sspace^1$};
  \draw[<-] (\bot)--(mid);
  \draw[<-] (\top)--(mid);
}
\end{scope}
\end{scope}
\begin{scope}[shift={(10,0)}]
\draw[rounded corners=1pt,fill=newcol] (-1,0) -- (-1.2,2) -- (-1.4,4) -- (-1,5) -- (-.8,4.5) -- (-.6,4) -- (-.3,3.5) -- (.4,3) -- (1.2,2.5) -- (1,2) -- (.8,2.5) -- (0,3) -- (.2,2) -- (0,1) -- (-.2,2) -- (-.4,3) -- (-.7,3.5) -- (-1,4) -- (-.8,2) -- (-1,0);
\foreach \x\y in {-1/2, -1/4.5, -.5/3.5, 0/2, 1/2.5}{
  \draw (\x-.2,\y) arc (-180:0:.2 and \circhy);
  \draw[densely dotted] (\x-.2,\y) arc (180:0:.2 and \circhy);
}
\foreach \x\y in {-1/4, 0/3}{
  \draw (\x,\y) arc (0:-180:.2 and \circhy);
  \draw (\x,\y) arc (-180:0:.2 and \circhy);
  \draw[densely dotted] (\x,\y) arc (0:180:.2 and \circhy);
  \draw[densely dotted] (\x,\y) arc (180:0:.2 and \circhy);
}
\node at (0,-1) {realized embedding of $\Sspace^2$};
\end{scope}
\end{tikzpicture} 
\caption{Constructing a 2-sphere embedding from a barcode.}
\label{fig:preimage-realization}
\end{figure}
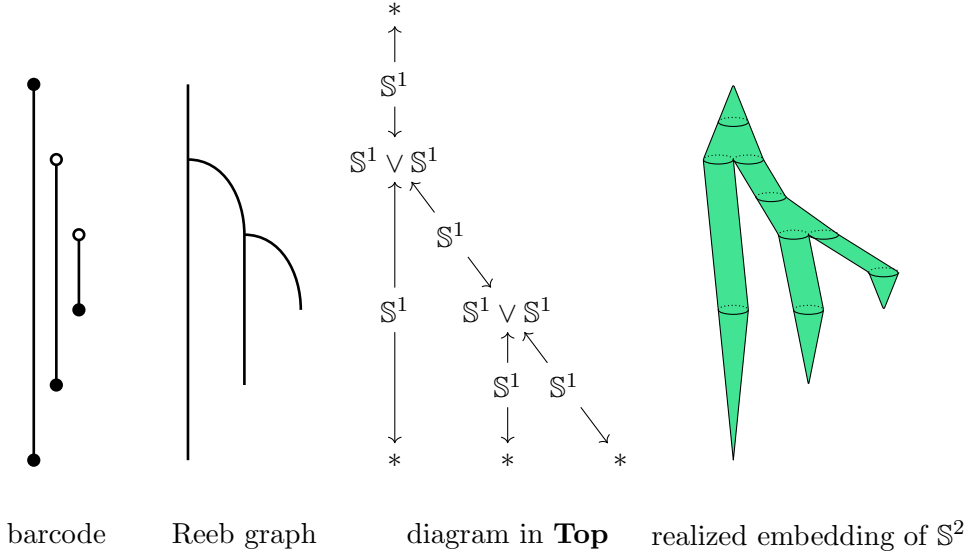

This construction suggests an algorithmic approach to constructing embeddings $\iota\colon \Sspace^2 \rightarrow \Rspace^3$ such that the barcode of the height function on $\im(\iota)$ will produce the barcode $B$. That is, every vertex of the Reeb graph of degree $n\geqslant 3$ corresponds to a wedge of $n-1$ spheres, every vertex of degree 1 corresponds to a point, and every edge between vertices corresponds to a zigzag $X\leftarrow \Sspace^1 \rightarrow Y$ in $\mathbf{Top}$. We leave the formalization and extension of these ideas open for further research.

\subsection{Extending the Nesting Poset}

Theorem \ref{theorem:nesting-poset-B} and its proof used arguments based on including smaller topological spaces into larger ones. Using the critical values of $f$, every open interval of $\Rspace$ containing at most one critical value can be uniquely associated with a nesting poset. Then the maps described in Theorem \ref{theorem:nesting-poset-B} correspond to restriction to a subset or containment in a superset, the former case described in Figure \ref{fig:sheaf-example}, following Figure \ref{fig:nesting3}.

\begin{figure}[h]\centering
\begin{tikzpicture}
\node at (-2,0) {$\Rspace$};
\draw[line width=1pt] (-1,0)--(7,0);
\foreach \x in {0,3,6}{
  \fill[white] (\x,0) circle (0.1);
  \fill (\x,0) circle (0.05);
}
\node[anchor=3] at (-1,0) {$\cdots$};
\node[anchor=177] at (7,0) {$\cdots$};
\draw[{Arc Barb[arc=100]}-{Arc Barb[arc=100]}] (2,.2) to node[above] {$U$} (2.8,.2);
\draw[{Arc Barb[arc=100]}-{Arc Barb[arc=100]}] (3.2,.2) to node[above] {$V$} (4,.2);
\draw[{Arc Barb[arc=100]}-{Arc Barb[arc=100]}] (1.8,-.2) to node[below] {$W$} (4.2,-.2);
\begin{scope}[shift={(.8,1)}]
\coordinate (A) at (.5,0);
\coordinate (Ab) at (-.5,.4);
\coordinate (Ac) at (-.5,-.4);
\fill (A) circle (.08);
\fill (Ab) circle (.08);
\fill (Ac) circle (.08);
\draw[conx] (Ab)--(A);
\draw[conx] (Ac)--(A);
\node[anchor=east,xshift=-.5cm] at ($(Ab)!.5!(Ac)$) {$\mathcal F(U) = $};
\coordinate (u) at ($(Ac)+(-45:.5)$);
\end{scope}
\begin{scope}[shift={(5,1)}]
\coordinate (A) at (.5,0);
\coordinate (b) at (-.5,0);
\fill (A) circle (.08);
\fill (b) circle (.08);
\draw[conx] (b)--(A);
\node[anchor=west,xshift=.2cm] at (A) {$= \mathcal F(V)$};
\coordinate (v) at ($(b)+(-35:.8)$);
\end{scope}
\begin{scope}[shift={(2.25,-1.3)}]
\coordinate (A) at (.5,0);
\coordinate (Ab) at (-.5,.4);
\coordinate (Ac) at (-.5,-.4);
\fill (A) circle (.08);
\fill (Ab) circle (.08);
\fill (Ac) circle (.08);
\draw[conx] (Ab)--(A);
\draw[conx] (Ac)--(A);
\node[anchor=west,xshift=.2cm] (ww) at (A) {$= \mathcal F(W)$};
\coordinate (w) at ($($(Ab)!.5!(Ac)$)+(180:.5)$);
\end{scope}
\draw[{Hooks[left]}->] (w) to [out=180,in=270] (u);
\draw[->>] (ww.east) to [out=0,in=270] (v);
\end{tikzpicture}
\caption{Associating poset maps to set restrictions.}
\label{fig:sheaf-example}
\end{figure}
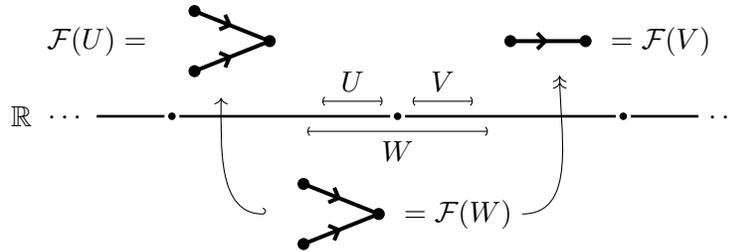

What we have described suggests the structure of an \emph{$\Rspace$-constructible sheaf} encoding the nesting poset. More detail is given by \cite{currypatel} as well as by \cite{patelmacpherson}
but we do not explore this here and leave it open for further research.

\subsection{Keeping Track of Critical Values}
\label{sec:tracking-values}

\edits{The techniques we have described here are concerned with the relative order of critical values, and only as a consequence of other structural changes. Using the order of critical values as a primary motivator opens up new directions of research, some of which} are addressed by ongoing work in vector field design~\cite{ZhouLazovskisCatanzaro2019} \edits{and existing literature on distances for topological invariants~\cite{difabiolandi2016,bauerlandimemoli2018,damicofrosinilandi2010}}.

In addition, specifying function values on singularities would allow for an additional measure on how ``far apart" two functions are: one could use the number of moves together with the difference of function values to measure their differences. The ``complexity" of a Morse function $f$ could then be given by measuring the difference between $f$ and a baseline function, such as the height function $h$ on the standard embedding of $\Sspace^2$ in $\Rspace^3$.


\section*{Acknowledgments}
This paper grew out of a productive discussion during the special workshop ``Bridging Statistics and Sheaves'' at the Institute for Mathematics and Applications in May 2018. The authors would like to thank the organizers for putting together the workshop, the IMA for hosting the event, and Mikael Vejdemo-Johansson for insightful conversations at the onset of this collaboration. The authors also thank the reviewers for helpful comments and suggestions. JC is partially funded by NSF CCF-1850052. JL is partially funded by EP/P025072/1. BTF is partially funded by NSF CCF-1618605 and DMS-1664858. BW is partially funded by NSF IIS-1513616 and IIS-1910733.

\bibliographystyle{abbrv}


\end{document}